\documentclass{article}
\usepackage[utf8]{inputenc}

\usepackage{amsthm}
\usepackage{amsmath}
\usepackage{amssymb}
\usepackage{url}
\usepackage{float}
\usepackage{bm}

\usepackage{graphicx}
\usepackage{mathtools}

\usepackage{hyperref}
    \hypersetup{
    colorlinks=true,%
    citecolor=blue,%
    filecolor=black,%
    linkcolor=blue,%
    urlcolor=black,%
    }
\usepackage{verbatim}
\usepackage{listings}
\usepackage{matlab-prettifier}
\usepackage{todonotes}

\DeclareMathOperator*{\argmax}{arg\,max}
\DeclareMathOperator*{\argmin}{arg\,min}

\newcommand{\Q}{\mathcal{Q}}

\renewcommand{\H}{\mathcal{H}}
\newcommand{\R}{\mathbb{R}}

\usepackage[ruled,vlined]{algorithm2e}

\newcommand{\G}{\mathcal{G}}

\newcommand{\Proj}{\operatorname{proj}}
\newcommand{\ran}{\operatorname{ran}}
\newcommand{\dom}{\operatorname{dom}}

\newcommand{\Id}{\operatorname{Id}}

\newcommand{\N}{\mathbb{N}}

\newtheorem{proposition}{Proposition}[section]
\theoremstyle{remark}
\newtheorem{remark}{Remark}[section]
\theoremstyle{definition}

\usepackage{graphicx}
\usepackage{fancyhdr} 
\usepackage[top=3cm,bottom=3cm,left=3.2cm,right=3.2cm,headsep=10pt]{geometry}
\usepackage[ruled,vlined]{algorithm2e}
\title{Splitting Algorithms for Distributionally Robust Optimization}
\author{Luis Briceño-Arias\footnote{Departamento de Matemática, Universidad tecnica federico santa maria, Santiago, Chile. E-mail: \href{mailto:luis.briceno@usm.cl}{luis.briceno@usm.cl}},
Sergio López-Rivera\footnote{Departamento de Ingeniería Matemática, Universidad de Chile, Santiago, Chile.
		E-mail: \href{mailto:sergio.lopez@dim.uchile.cl}{sergio.lopez@dim.uchile.cl}} and 
  Emilio Vilches\footnote{Instituto de Ciencias de la Ingeniería, Universidad de O’Higgins, Rancagua, Chile.
		E-mail: \href{mailto:emilio.vilches@uoh.cl}{emilio.vilches@uoh.cl}}}
\date{\today}

\makeatletter
\newcommand{\otherlabel}[2]{\protected@edef\@currentlabel{#2}\label{#1}}
\makeatother

\begin{document}

 \maketitle
 \begin{abstract}
In this paper, we provide different splitting methods for solving distributionally robust optimization problems in cases where the uncertainties are described by discrete distributions. The first method 
 involves computing the proximity operator of the supremum function that appears in the optimization problem. The second method solves an equivalent monotone inclusion formulation derived from the first-order optimality conditions, where the resolvents of the monotone operators involved in the inclusion are computable. The proposed methods are applied to solve the Couette inverse problem with uncertainty and the denoising problem with uncertainty. We present numerical results to compare the efficiency of the algorithms. 
 
 \end{abstract}
	\section{Introduction}
Stochastic optimization is a method for solving decision problems that involve uncertainty or randomness. The goal is to find solutions that work well under different conditions by optimizing an objective function that accounts for the uncertainty, often focusing on expected results or reducing risks (see, e.g., \cite{MR4362585}). However, this approach usually assumes full knowledge of the probability distributions, which may not always be practical in real situations. To address this situation, the theoretical framework of Distributionally Robust Optimization (DRO) has been introduced, which assumes that the true probability distribution belongs to a set of distributions known as the \emph{ambiguity set}. Optimization is then performed based on the worst-case scenario within this ambiguity set. By considering different ambiguity sets, DRO encompasses Robust Optimization when the ambiguity set includes all possible distributions, and Stochastic Optimization when the ambiguity set consists of a single distribution. Hence, the DRO framework is well-suited for addressing problems with partial information, promoting distributed robustness in decision-making. We refer to \cite{MR4362585,sun2021robust,wiesemann2014distributionally} for more details. 

   Let $(\Omega, \mathcal{A})$ be a measurable space, and let $\mathcal{P}$ be a nonempty, closed and convex subset of probability measures defined on $(\Omega, \mathcal{A})$, supported on $\Xi \subset \mathcal{G}$, where $\mathcal{G}$ is a Hilbert space. In this paper, we aim to study splitting methods for solving the following distributionally robust optimization problem:
	\begin{align}
		\label{dro_ver_cont}
		\min_{x\in \Q} \left\{h(x)+\sup_{\mathbb{P} \in \mathcal{P}}\mathbb{E}_{{\mathbb{P}}} \left[F(x,\xi)\right]\right\},
	\end{align}
	where $\Q$ is a nonempty convex closed subset of a real Hilbert space $\H$ and $h\colon\H\rightarrow\R$ is a convex and differentiable function with $\beta^{-1}$-Lipschitz gradient for some $\beta>0$. Besides, $\xi\colon \Omega \to \Xi$ is a random vector, and $F\colon \mathcal{H} \times \Xi \to \mathbb{R}$ is the random cost function, with $F(\cdot, \xi)$ being proper, lower semicontinuous, and convex for every $\xi \in \Xi$. 
The set $\mathcal{P}$ is called the \emph{ambiguity set} and accounts for the level of knowledge about the probability model of the problem. This set can be constructed based on empirical statistical information or beliefs about the moments of the distribution itself (see \cite{sun2021robust}). Moreover, when the ambiguity set $\mathcal{P}$ is a singleton, the DRO model \eqref{dro_ver_cont} reduces to a stochastic optimization problem. Conversely, if $\mathcal{P}$ represents the set of all probability measures defined on $(\Omega,\mathcal{A})$ and supported on $\Xi$, then problem~\eqref{dro_ver_cont} becomes a robust optimization problem. Thus, the DRO model generalizes both robust and stochastic optimization, offering a unified framework where robust and stochastic optimization are two extreme cases. 

It is well-known that both robust and stochastic optimization have certain limitations. First, robust optimization can be very conservative, as it ignores valuable probabilistic information, while stochastic optimization may require too much information about the probability distribution, which may not be available to the modeler. Second, the solutions provided by robust or stochastic optimization models may perform poorly in out-of-sample tests or may have intrinsic bias that cannot be eliminated simply by increasing the size of the sampled data. Third, robust models can be computationally difficult to solve. Additionally, stochastic programs may involve high-dimensional integration, which is also generally intractable. In this context, the DRO problem provides a potent modeling framework and addresses some of the disadvantages of both stochastic and robust optimization. We refer to \cite{MR4362585} for a thorough discussion.

 One way to tackle the problem~\eqref{dro_ver_cont} is to use the techniques proposed in \cite[Section~4]{sun2021robust} or in \cite[Section~7]{MR4362585}, that is, formulate the dual program of the inner maximization problem:
 \begin{align}
\label{primal_ex_1}
\max_{\mathbb{P}\in\mathcal{P}}\mathbb{E}_{\mathbb{P}}[\overline{F}(x,\xi)]:=\mathbb{E}_{\mathbb{P}}[h(x)+F(x,\xi)].
 \end{align}
Let us illustrate this technique considering $\G=\R^{m}$ and the following ambiguity set \cite[eq. 4.2.1]{sun2021robust}
 \begin{align}
\label{amb_set_1}
\mathcal{P}(\Xi, \bm{\mu}^{+}, \bm{\mu}^{-})=\{\mathbb{P}\in\mathcal{M}_{+}(\mathcal{A})\,:\,\mathbb{P}(\xi\in\Xi)=1,\,\,\bm{\mu}^{-}\leq\mathbb{E}_{\mathbb{P}}[\xi]\leq\bm{\mu}^{+}\},
\end{align}
where $\mathcal{M}_{+}(\mathcal{A})$ is the set of all the measures defined on $\mathcal{A}$, $\bm{\mu}^{+}\in\R^{m}$, and $\bm{\mu}^{-}\in\R^{m}$. Denote by $\mathcal{M}_{+}(\R^{m})$ the set of all the measures induced by the random vector $\xi$. That is, $\mu\in\mathcal{M}_{+}(\R^{m})$ if $\mu(A)=\mathbb{P}(\xi^{-1}(A))$, for all $A\subset\R^{m}$ Borel measurable set, where $\mathbb{P}\in\mathcal{M}_{+}(\mathcal{A})$. Then, the problem \eqref{primal_ex_1} is equivalent to the following maximization problem
\begin{align}
\label{primal_ex_2}
\max_{\mu\in\mathcal{M}_{+}(\R^{m})}\,\, &\int_{\Xi}\overline{F}(x,\xi)d\mu(\xi)\\
\text{s.t. }&\int_{\Xi}d\mu(\xi)=1,\quad\int_{\Xi}\xi d\mu(\xi)\leq \bm{\mu}^{+},\quad\int_{\Xi}-\xi d\mu(\xi)\leq -\bm{\mu}^{-},\nonumber
\end{align}
whose dual problem \eqref{primal_ex_2} is given by
\begin{align}
&\min_{\lambda\in\R,\bm{\beta}\geq 0,\bm{\gamma}\geq 0}\lambda+\bm{\beta}^{\top}\bm{\mu}^{+}-\bm{\gamma}^{\top}\bm{\mu}^{-}\label{dual_ex_1}\\
&\qquad\text{s.t. }\overline{F}(x,\xi)-\lambda+(\bm{\gamma}-\bm{\beta})^{\top}\xi\leq 0,\quad\forall \xi\in\Xi.\nonumber
\end{align}
If the strong duality holds between the primal problem~\eqref{primal_ex_2} and its dual \eqref{dual_ex_1} (for example, if the Slater conditions hold for \eqref{primal_ex_2}), we can replace the inner maximization problem with its dual in the original DRO problem~\eqref{dro_ver_cont}. Thus, the problem \eqref{dro_ver_cont} is equivalent to the following static robust optimization problem
\begin{align}
\label{equiv_dro_1}
\min_{x\in\H, \lambda\in\R, \bm{\beta}\geq 0, \bm{\gamma}\geq 0}&\,\, \lambda+\bm{\beta}^{\top}\bm{\mu}^{+}-\bm{\gamma}^{\top}\bm{\mu}^{-}\\
\text{s.t. }&\overline{F}(x,\xi)-\lambda+(\bm{\gamma}-\bm{\beta})^{\top}\xi\leq 0,\quad\forall \xi\in\Xi,\nonumber\\
&x\in \Q.\nonumber
\end{align}
The difficulty of the program~\eqref{equiv_dro_1} depends on the structure of the function $\overline{F}(x,\xi)$ as well as of the set $\Q$. We focus on the discrete case, that is, when the set $\Xi$ associated to $\xi$ is finite: $\Xi=\{\xi_{1},\ldots,\xi_{N}\}\subset \G$ with $N\geq 1$. For every $\mathbb{P}\in\mathcal{P}$ and $i\in\{1,\ldots,N\}$, we denote $p_{i}:=\mathbb{P}(\{\omega\in\Omega\,:\,\xi(\omega)=\xi_{i}\})$. Then, the set of all the probability measures is given by the following set (the probability simplex)
	\begin{align*}
		\Delta_{N}:=\left\{p\in\R_{+}^{N}\,:\,\sum_{i=1}^{N}p_{i}=1 \right\}.
	\end{align*}
 and $\mathbb{E}_{\mathbb{P}}\left[F(x,\xi)\right]=\sum_{i=1}^{N}p_{i}F(x,\xi_{i})$. Hence, the problem \eqref{dro_ver_cont} becomes
 \begin{align}
		\min_{x\in\Q} \left\{h(x)+\sup_{p\in \mathcal{P}} \sum_{i=1}^{N}p_{i}f_{i}(x)\right\},
  \label{prob_principal}
	\end{align}
 where $f_i:=F(\cdot,\xi_i)$ for all $i\in \{1,\ldots,N\}$. As far as we know, for the above problem, there are no available splitting algorithms. Hence, the aim of this paper is twofold. First, to propose a new algorithm for solving the following generalization of problem~\eqref{dro_ver_cont} in the discrete case:
  \begin{align}
\label{gen_dro_disc}
\min_{\mathbf{x}=(x_{1},\ldots,x_{N})\in\H^{N}}\left\{H(\mathbf{x})+\sup_{p\in\mathcal{P}}\sum_{i=1}^{N}p_{i}f_{i}(x_{i})\right\}\quad \text{s.t.}\quad\mathbf{x}\in V\cap\bigtimes_{i=1}^{N}\Q_{i},
 \end{align}
 where $H\colon\H^{N}\rightarrow\R$ is a convex differentiable function with $\beta^{-1}$-Lipschitz gradient, $\Q_{i}\subset\H$ is a nonempty closed convex set, $V\subset\H^{N}$ is a closed vector subspace, and $\mathcal{P}\subset \Delta_{N}$ is a nonempty closed convex set. 
 Second, to propose different methods to solve the problem \eqref{prob_principal}. \vspace{0.1cm} 

 On the hand, we observe that whenever $V=\mathcal{D}:=\{\mathbf{x}\in\H^{N}\,:\,x_{1}=\cdots=x_{N}\}$, $H(\mathbf{x})=h(x_{1})$, and $\Q_{i}=\Q$ for all $i\in\{1,\ldots,N\}$, we have that $(x,\ldots,x)\in\mathcal{D}$ is solution to problem~\eqref{gen_dro_disc} if and only if $x\in\H$ is a solution to \eqref{prob_principal}. On the other hand, when $\H=\R^{n}$, $V$ is the nonanticipativity set, and $H=0$, problem~\eqref{gen_dro_disc} reduces to the problem considered in \cite[Example~4]{de2021risk}. In order to solve problem \eqref{gen_dro_disc}, we consider the following functions $(f_{i})_{i=1}^{N}$ and ambiguity sets $\mathcal{P}$. For every $i\in\{1,\ldots,N\}$,
 \begin{enumerate}
 \item\label{item1} $f_{i}(x)=\langle a_{i},x\rangle+\xi_{i}$, where $a_{i}\in \H\backslash\{0\}$ and $\xi_{i}\in\R$.
  \begin{enumerate}
 \item\label{item1_1} $\mathcal{P}=\Delta_{N}$.
     \item\label{item1_2} $\mathcal{P}=\{p\in\Delta_{N}\,:\,p\leq q\}$, where $q\in\R^{N}$ is such that $\text{int}(\R^{N}_{+})\cap\text{int}( \mathcal{P}_{2})\cap\mathcal{P}_{1}\neq\emptyset$ with $\mathcal{P}_{1}=\{p\in\R^{N}\,:\,\sum_{i=1}^{N}p_{i}=1\}$ and $\mathcal{P}_{2}=\{p\in\R^{N}\,:\,p\leq q\}$.
     \item\label{item1_3} $\mathcal{P}=\{p\in\Delta_{N}\,:\,\mu_{-}\leq \sum_{i=1}^{N}p_{i}\xi_{i}\leq\mu_{+}\}$, where $\mu_{-}\in\R$, and $\mu_{+}\in\R$ are such that $\text{int}(\R^{N}_{+})\cap\text{int}( \mathcal{P}_{3})\cap\mathcal{P}_{1}\neq\emptyset$ with  $\mathcal{P}_{3}=\{p\in\R^{N}\,:\,\mu_{-}\leq \sum_{i=1}^{N}p_{i}\xi_{i}\leq\mu_{+}\}$.
     \end{enumerate}
 \item\label{item2} $f_{i}(x)=\|x-\xi_{i}\|^{2}$, where $\xi_{i}\in\H$, and $\mathcal{P}=\Delta_{N}$.
\end{enumerate}
 \begin{remark}
 Note that in the context of problem~\eqref{prob_principal}, the case when $f_{i}(x)=\langle x,Qx\rangle+\langle b_{i},x\rangle+c_{i}$, where $Q\colon\H\rightarrow\H$ is a symmetric positive semidefinite operator, $b_{i}\in\H$, and $c_{i}\in\R$, it reduces to the case \ref{item1}. Indeed, since $\mathcal{P}\subset\Delta_{N}$, we have in this case that
 \begin{align*}
h(x)+\sup_{p\in\mathcal{P}}\sum_{i=1}^{N}p_{i}f_{i}(x)=\underbrace{h(x)+\langle x,Qx\rangle}_{\widetilde{h}(x)}+\sup_{p\in\mathcal{P}}\sum_{i=1}^{N}p_{i}(\langle b_{i},x\rangle+c_{i}),
 \end{align*}
 where $\widetilde{h}$ is convex differentiable with Lipschitz gradient.
 \end{remark}
 
 The case~\ref{item1_2} covers the case when the ambiguity set is associated to the conditional value-at-risk, which is
\begin{align}
\label{set_cvar}
    \mathcal{P}_{\text{CVaR}_{\alpha}}:=\left\{p\in\Delta_{N}\,:\,p\leq \frac{\overline{p}}{1-\alpha}\right\},
\end{align}
where $\overline{p}\in\R^{N}$ is a probability vector and $\alpha\in\left]0,1\right[$. The ambiguity set in \eqref{set_cvar} appears in \cite{de2021risk}. On the other hand, the case~\ref{item1_3} is motivated by the ambiguity set defined in \cite[eq. 4.2.1]{sun2021robust}.
\noindent In the case~\ref{item1}, problem~\eqref{gen_dro_disc} is
 \begin{align}
\label{prob_item2}
\min_{\mathbf{x}\in\H^{N}}\left\{H(\mathbf{x})+\sup_{p\in\mathcal{P}}\sum_{i=1}^{N}p_{i}(\langle a_{i},x_{i}\rangle+\xi_{i})\right\}\quad \text{s.t.}\quad\mathbf{x}\in V\cap\bigtimes_{i=1}^{N}\Q_{i}.
 \end{align}
 In the case~\ref{item2}, problem~\eqref{gen_dro_disc} reduces to
 \begin{align}
\label{prob_item3}
\min_{\mathbf{x}\in\H^{N}}\left\{H(\mathbf{x})+\max_{1\leq i\leq N}\|x_{i}-\xi_{i}\|^{2}\right\}\quad \text{s.t.}\quad\mathbf{x}\in V\cap\bigtimes_{i=1}^{N}\Q_{i}.
 \end{align}
On the other hand, in the case~\ref{item1}, problem~\eqref{prob_principal} is equivalent to
\begin{align}
\label{prob_item4}
\min_{x\in\Q} \left\{h(x)+\sup_{p\in\mathcal{P}}\sum_{i=1}^{N}p_{i}(\langle a_{i},x\rangle+\xi_{i})\right\}.
\end{align}
 
  \paragraph{Notation:} Let $\H$ a real Hilbert space with inner product $\langle\cdot,\cdot\rangle$ and induced norm $\|\cdot\|$.
For every $\lambda>0$, the Moreau-Yosida regularization of $f\colon\H\rightarrow\R\cup\{+\infty\}$ of parameter $\lambda$ is denoted by
\begin{align*}
x\in\H\mapsto e_{\lambda}f(x):=\inf_{u\in\H}\left\{f(u)+\frac{1}{2\lambda}\|x-u\|^{2}\right\}
\end{align*}
and the proximity operator is denoted by
 \begin{align*}
x\in\H\mapsto P_{\lambda}f(x)=\argmin_{y\in\H}\left\{f(y)+\frac{1}{2\lambda}\|x-y\|^{2}\right\}.
 \end{align*}
 Moreover, the projection on a set $\Q\subset \H$ is denoted by
 \begin{align*}
 x\in\H\mapsto \Proj_{\Q}(x)=\argmin_{y\in \Q}\|x-y\|.
 \end{align*}
 The indicator function of a set $\Q$ is denoted by $\iota_{\Q}$, which is equal to zero in $\Q$ and $+\infty$ otherwise. On the other hand, the normal cone to a closed convex set $\Q$ is denoted by $N_{\Q}$ and the interior of $\Q$ is denoted by $\text{int}(\Q)$. Moreover, the set of the functions $f\colon\H\rightarrow\R\cup\{+\infty\}$ that are proper, lower semicontinuous, and convex is denoted by $\Gamma_{0}(\H)$. The convex subdifferential of a function $f\in\Gamma_{0}(\H)$ is denoted by $\partial f$. In addition, the domain of $f$ is denoted by $\dom f=\{x\in\H\,:\,f(x)<+\infty\}$. The conjugate of a function $f\in\Gamma_{0}(\H)$ is denoted by $f^{*}$ and the resolvent of a maximally monotone operator $A$ is denoted by $J_{A}$. Now, the class of bounded linear operators from $\H$ to a real Hilbert space $\mathcal{G}$ is denoted by $\mathcal{L}(\H,\mathcal{G})$ and if $\H=\mathcal{G}$ this class is denoted by $\mathcal{L}(\H)$. Given $L\in\mathcal{L}(\H,\mathcal{G})$, its adjoint operator is denoted by $L^{*}\in\mathcal{L}(\mathcal{G},\H)$. Finally, the expected value of a random variable $X\colon\Omega\rightarrow\Xi$ with respect to a probability measure $\mathbb{P}$ is denoted by $\mathbb{E}_{\mathbb{P}}[X]=\int_{\Omega}X(\omega)d\mathbb{P}(\omega)$.  

 The remainder of this paper is organized as follows. In Section~\ref{sec_prox_sup}, we compute the proximity operator of certain supremum functions. In some cases, we provide an algorithm that converges to the proximity operator, while in a particular case we provide a closed-form for the proximity operator. Next, in Section~\ref{sec_alg}, we will show the different algorithms for solving the problems described in the introduction. In Section~\ref{sec_app}, we provide some applications that can be written as a particular case of the main problem. In Section~\ref{sec_num_exp}, we illustrate the numerical experiments of the proposed algorithms. Finally, conclusions are detailed in Section~\ref{sec_conc}.


\section{Proximity of a supremum function}
\label{sec_prox_sup}
Consider the following function
 \begin{align}
 \label{fun_sup_sep_1}
\mathbf{x}\in\H^{N}\mapsto f(\mathbf{x}):=\sup_{p\in\mathcal{P}}\sum_{i=1}^{N}p_{i}f_{i}(x_{i}).
 \end{align}
 In the following result, we compute the proximity operator of $f$ in the cases \ref{item1} and \ref{item2}. We will see that in the case~\ref{item2}, the proximity of $f$ has a closed-form.
  \begin{proposition}
\label{calc_prox_sup}
Let $\{f_{i}\}_{i=1}^{N}\subset \Gamma_{0}(\H)$ and let $\mathcal{P}\subset\Delta_{N}$ be a nonempty closed convex set. Let $\mathbf{x}\in\H^{N}$ and $\lambda>0$. Then the following hold.
\begin{enumerate}
    \item\label{case_1} If for every $i\in\{1,\ldots,N\}$, $f_{i}(x)=\langle a_{i},x\rangle+\xi_{i}$ with $a_{i}\in\H\backslash\{0\}$ and $\xi_{i}\in\R$, then
    \begin{align}
    \label{prox_sup_aff}
    P_{\lambda}f(\mathbf{x})=\left(x_{i}-\lambda\overline{p}_{i}a_{i}\right)_{i=1}^{N},
    \end{align}
    where $\overline{p}\in\mathcal{P}$ is the unique solution to
     \begin{align}
 \label{prob_aux_1}
\min_{p\in\mathcal{P}} \dfrac{1}{2}p^{\top}Dp-p^{\top}\beta,
 \end{align}
 with $D:=\textnormal{diag}(\lambda\|a_{1}\|^{2},\ldots,\lambda\|a_{N}\|^{2})$ and $\beta:=(\langle a_{i},x_{i}\rangle+\xi_{i})_{i=1}^{N}$.
 \item\label{case_2} If $\mathcal{P}=\Delta_{N}$ and for every $i\in\{1,\ldots,N\}$, $f_{i}(x)=\|x-\xi_{i}\|^{2}$ with $\xi_{i}\in\H$, then
 \begin{align}
 \label{prox_sup_quad}
P_{\lambda}f(\mathbf{x})=\left(\frac{x_{i}+2\lambda \overline{p}_{i}\xi_{i}}{2\lambda \overline{p}_{i}+1}\right)_{i=1}^{N},
 \end{align}
 where $\overline{p}\in\Delta_{N}$ is a solution to
 \begin{align}
\label{prob_aux_2}
\max_{p\in\Delta_{N}} \ell(p):=\sum_{i=1}^{N}\left(\frac{p_{i}}{1+2\lambda p_{i}}\right)\alpha_{i},
\end{align}
and, for every $i\in\{1,\ldots,N\}$, $\alpha_{i}=\|x_{i}-\xi_{i}\|^{2}$.
\end{enumerate}
  \end{proposition}
  \begin{proof}
Let us define the function $\mathbf{x}\in\H^{N}\mapsto g_{p}(\mathbf{x})=\sum_{i=1}^{N}p_{i}f_{i}(x_{i})$ for all $p\in\mathcal{P}$. Note that, since $f_{i}\in\Gamma_{0}(\H)$ for all $i\in\{1,\ldots,N\}$, then $\{g_{p}\}_{p\in\mathcal{P}}\subset\Gamma_{0}(\H^{N})$. Moreover, the function $p\mapsto g_{p}(\mathbf{x})$ is concave and upper semicontinuous for all $\mathbf{x}\in\H^{N}$. In addition, since $f_{i}$ is proper for all $i\in\{1,\ldots,N\}$ and $\mathcal{P}\subset\Delta_{N}$, then $f=\sup_{p\in\mathcal{P}}g_{p}$ is proper. Furthermore, $\mathcal{P}$ is a nonempty compact and convex set.
Then, by \cite[Theorem~3.5]{MR4279933}, we have that 
  \begin{align}
   \label{prox_sup_sep}
P_{\lambda}f(\mathbf{x})=P_{\lambda}g_{\overline{p}}(\mathbf{x}),\quad\text{ with } \overline{p}\in \displaystyle\argmax_{p\in\mathcal{P}}e_{\lambda}g_{p}(\mathbf{x})
  \end{align}
  Case \ref{case_1}: By \cite[Proposition~24.8(i) \& Proposition~24.11]{MR3616647}, we deduce that
 \begin{align}
\label{prox_gp_2}
P_{\lambda}g_{\overline{p}}(\mathbf{x})=\left(x_{i}-\lambda\overline{p}_{i}a_{i}\right)_{i=1}^{N}.
 \end{align}
 Let us compute $e_{\lambda}g_{p}(\mathbf{x})$. By \eqref{prox_gp_2}, we obtain that
 \begin{align*}
e_{\lambda}g_{p}(\mathbf{x})&=g_{p}(P_{\lambda}g_{p}(\mathbf{x}))+\frac{1}{2\lambda}\|\mathbf{x}-P_{\lambda}g_{p}(\mathbf{x})\|^{2}\\
			&=\sum_{i=1}^{N}p_{i}(\langle a_{i},x_{i}-\lambda p_{i}a_{i}\rangle+\xi_{i})+\frac{1}{2\lambda}\sum_{i=1}^{N}\|x_{i}-x_{i}+\lambda p_{i}a_{i}\|^{2}\\
			&=\sum_{i=1}^{N}p_{i}(\langle a_{i},x_{i}\rangle+\xi_{i})-\lambda\sum_{i=1}^{N}p_{i}^{2}\|a_{i}\|^{2}+\frac{1}{2\lambda}\sum_{i=1}^{N}\lambda^{2}p_{i}^{2}\|a_{i}\|^{2}\\
			&=\sum_{i=1}^{N}p_{i}(\langle a_{i},x_{i}\rangle+\xi_{i})-\frac{\lambda}{2}\sum_{i=1}^{N}p_{i}^{2}\|a_{i}\|^{2}.
 \end{align*}
 Set $\alpha_{i}=\lambda\|a_{i}\|^{2}$ and $\beta_{i}=\langle a_{i},x_{i}\rangle+\xi_{i}$ for all $i\in\{1,\ldots,N\}$. Then, in order to find $\overline{p}$ in \eqref{prox_sup_sep}, we need to solve the following problem
 \begin{align*}
\max_{p\in\mathcal{P}}\sum_{i=1}^{N}p_{i}\beta_{i}-\dfrac{1}{2}\sum_{i=1}^{N}p_{i}^{2}\alpha_{i},
 \end{align*}
 which is equivalent to
 \begin{align}
 \label{prob_p_bar_2}
\min_{p\in\mathcal{P}} \psi(p):=\dfrac{1}{2}p^{\top}Dp-p^{\top}\beta,
 \end{align}
 where $D=\text{diag}(\alpha_{1},\ldots,\alpha_{N})\in\R_{++}^{N\times N}$ and $\beta=(\beta_{i})_{i=1}^{N}\in\R^{N}$. Note that, since $\nabla^{2}\psi(p)=D$ is positive definite, $\psi$ is strictly convex and coercive. Thus, from \cite[Proposition~11.15(i)]{MR3616647}, we have that the problem~\eqref{prob_p_bar_2} has an unique solution.

 Case \ref{case_2}: By \cite[Proposition~24.11 \& Proposition~24.8(i)]{MR3616647}, we obtain that
 \begin{align}
 \label{prox_gp_1}
P_{\lambda}g_{\overline{p}}(\mathbf{x})=\left(\frac{1}{2\lambda \overline{p}_{i}+1}(x_{i}+2\lambda \overline{p}_{i}\xi_{i})\right)_{i=1}^{N}.
 \end{align}
Let us calculate $e_{\lambda}g_{p}(\mathbf{x})$. By \eqref{prox_gp_1}, we have that
\begin{align*}
e_{\lambda}g_{p}(\mathbf{x})&=g_{p}(P_{\lambda}g_{p}(\mathbf{x}))+\frac{1}{2\lambda}\|\mathbf{x}-P_{\lambda}g_{p}(\mathbf{x})\|^{2}\\
&=\sum_{i=1}^{N}p_{i}\left\|\frac{x_{i}+2\lambda p_{i}\xi_{i}}{1+2\lambda p_{i}}-\xi_{i}\right\|^{2}+\frac{1}{2\lambda}\sum_{i=1}^{N}\left\|\frac{x_{i}+2\lambda p_{i}\xi_{i}}{1+2\lambda p_{i}}-x_{i}\right\|^{2}\\
&=\sum_{i=1}^{N}p_{i}\left\|\frac{x_{i}-\xi_{i}}{1+2\lambda p_{i}}\right\|^{2}+\frac{1}{2\lambda}\sum_{i=1}^{N}\left\|\frac{2\lambda p_{i}\xi_{i}-2\lambda p_{i}x_{i}}{1+2\lambda p_{i}}\right\|^{2}\\
&=\sum_{i=1}^{N}\frac{p_{i}\|x_{i}-\xi_{i}\|^{2}}{(1+2\lambda p_{i})^{2}}+\frac{(2\lambda p_{i})^{2}\|\xi_{i}-x_{i}\|^{2}}{2\lambda(1+2\lambda p_{i})^{2}}\\
&=\sum_{i=1}^{N}\frac{p_{i}(1+2\lambda p_{i})}{(1+2\lambda p_{i})^{2}} \|x_{i}-\xi_{i}\|^{2}\\
&=\sum_{i=1}^{N}\left(\frac{p_{i}}{1+2\lambda p_{i}}\right)\|x_{i}-\xi_{i}\|^{2}.
\end{align*}
Set $\alpha_{i}=\|x_{i}-\xi_{i}\|^{2}$ for all $i\in\{1,\ldots,N\}$. Then, in order to find $\overline{p}$ in \eqref{prox_sup_sep}, we need to solve in this case the following problem
	\begin{align}
		\label{prob_paso_1}
		&\max_{p\in \Delta_{N}} \ell(p):=\sum_{i=1}^{N}\left(\frac{p_{i}}{1+2\lambda p_{i}}\right)\alpha_{i}.
	\end{align}
Note that $\ell$ is continuous and $\Delta_{N}$ is a nonempty compact set. Therefore, the problem~\eqref{prob_paso_1} has solutions.
 \end{proof}

 
  \begin{remark}
In the case when $\mathcal{P}=\Delta_{N}$ (case \ref{item1_1}) or $\mathcal{P}=\{p\in\Delta_{N}\,:\,p\leq q\}$ (case \ref{item1_2}), the problem in \eqref{prob_aux_1} can be solved by the method proposed in \cite{cominetti2014newton}.
  \end{remark}
  The following result provides a method for solving \eqref{prob_aux_1} when $\mathcal{P}=\{p\in\Delta_{N}\,:\,\mu_{-}\leq \sum_{i=1}^{N}p_{i}\xi_{i}\leq\mu_{+}\}$ (case \ref{item1_3}). Recall that $\mathcal{P}_{3}=\{p\in\R^{N}\,:\,\mu_{-}\leq \sum_{i=1}^{N}p_{i}\xi_{i}\leq\mu_{+}\}$.
  \begin{proposition}
\label{dykstra_prop_1}
In the context of problem~\eqref{prob_aux_1}, let $R:=\sqrt{D}$, let $x_{0}=R^{-1}\beta$, and let $p_{0}=q_{0}=0$. For every $k\in\mathbb{N}$, we consider the following routine
\begin{equation}
\label{dykstra_alg_1}
\left\lfloor
\begin{array}{ll}
y_{k}=\Proj_{R\mathcal{P}_{3}}(x_{k}+p_{k})\\
p_{k+1}=x_{k}+p_{k}-y_{k} \\
x_{k+1}=\Proj_{R\Delta_{N}}(y_{k}+q_{k})\\
q_{k+1}=y_{k}+q_{k}-x_{k+1}.
\end{array}
\right.
\end{equation}
Then $(x_{k})$ converges to a point $q$ and $p=R^{-1}q$ is the solution to \eqref{prob_aux_1} with $\mathcal{P}=\{p\in\Delta_{N}\,:\,\mu_{-}\leq \sum_{i=1}^{N}p_{i}\xi_{i}\leq\mu_{+}\}$.
  \end{proposition}
  \begin{proof}
Note that the problem~\eqref{prob_aux_1} with $\mathcal{P}=\{p\in\Delta_{N}\,:\,\mu_{-}\leq \sum_{i=1}^{N}p_{i}\xi_{i}\leq\mu_{+}\}$ is equivalent to
\begin{align}
\label{prob_aux_1_2}
\min_{p\in \Delta_{N}\cap \mathcal{P}_{3}}\psi(p):=\dfrac{1}{2}p^{\top}Dp-p^{\top}\beta.
\end{align}
Let us consider the variable change $q=Rp$. Then, since $R=\sqrt{D}$ and $D\in\R^{N\times N}_{++}$ is a diagonal matrix, we have 
\begin{align*}
\psi(q)&=\dfrac{1}{2}(R^{-1}q)^{\top}DR^{-1}q-(R^{-1}q)^{\top}\beta\\
&=\dfrac{1}{2}q^{\top}R^{-1}DR^{-1}q-q^{\top}R^{-1}\beta\\
&=\dfrac{1}{2}q^{\top}q-q^{\top}R^{-1}\beta=\dfrac{1}{2}\|q-R^{-1}\beta\|^{2}-\dfrac{1}{2}\|R^{-1}\beta\|^{2}.
\end{align*}
Thus, problem~\eqref{prob_aux_1_2} is equivalent to find 
\begin{align}
\label{proj_inter}
\overline{q}=\Proj_{R\Delta_{N}\cap R\mathcal{P}_{3}}(R^{-1}\beta).
\end{align}
From \cite[Proposition~5.3]{combettes2011proximal}, the sequence $(x_{k})$ generated by the algorithm in \eqref{dykstra_alg_1} (Dykstra's projection algorithm) converges to the projection $\overline{q}$ in \eqref{proj_inter}. Therefore, $p=R^{-1}\overline{q}$ is the solution to \eqref{prob_aux_1_2}.
  \end{proof}
  \begin{remark}
    In the context of problem~\eqref{prob_aux_1}, denote $D=\text{diag}(\alpha_{1},\ldots,\alpha_{N})$. Note that $R\Delta_{N}=\{q\in\R_{+}^{N}\,:\,\sum_{i=1}^{N}\frac{q_{i}}{\sqrt{\alpha_{i}}}=1\}$ and $R\mathcal{P}_{3}=\{q\in\R^{N}\,:\,\mu_{-}\leq\sum_{i=1}^{N}q_{i}\frac{\xi_{i}}{\sqrt{\alpha_{i}}}\leq\mu_{+}\}$. Then, the projection on $R\Delta_{N}$ can be calculated using the algorithm proposed in \cite{cominetti2014newton} and the projection on $R\mathcal{P}_{3}$ can be calculated through of \cite[Example~29.21]{MR3616647}.
  \end{remark}
	
Let us see now a method for solving the problem~\eqref{prob_aux_2}, where $\alpha_{i}\geq 0$ for all $i\in\{1,\ldots,N\}$. Note that if $\alpha_{i}=0$ for all $i\in\{1,\ldots,N\}$, then every $p\in\Delta_{N}$ is a solution to \eqref{prob_aux_2}. Thus, we can assume that $\displaystyle\max_{1\leq i\leq N}\alpha_{i}>0$. On the other hand, note that for every $p\in \Delta_{N}$, we have
\begin{align*}
	(\nabla^{2}\ell(p))_{ij}=\begin{cases}
		-4\lambda\alpha_{i}(1+2\lambda p_{i})^{-3} & \text{ if } i=j\\
		0 & \text{ if }i\neq j
	\end{cases}\quad\text{for all } i,j\in\{1,\ldots,N\}.
	\end{align*}
Then, $\nabla^{2}(-\ell)(p)$ is a positive semidefinite matrix for all $p\in\Delta_{N}$ and thus $-\ell$ is convex on $\Delta_{N}$. The following result provide an explicit solution to problem \eqref{prob_aux_2}.
	\begin{proposition}
		\label{prop_prob_paso_1}
		In the context of problem \eqref{prob_aux_2}, let $\{\ell_{i}\}_{i=1}^{N}$ such that $\alpha_{\ell_{1}}\leq\cdots\leq\alpha_{\ell_{N}}$ with $\alpha_{\ell_N}>0$ and let $A_{i}:=\{\ell_{1},\ldots,\ell_{i}\}$ for all $i\in\{1,\ldots,N\}$. Define
		\begin{align}
			\label{def_k}
			k:=\min\left\{i\in\{0,\ldots,N-1\}\,:\,(N-i+2\lambda)\sqrt{\alpha_{\ell_{i+1}}}>\sum_{j\notin A_{i}}\sqrt{\alpha_{j}}\right\},
		\end{align}
		where $A_{0}:=\emptyset$. Then $\overline{p}\in \R^{N}$ defined by
		\begin{align}
			\label{sol_prob_paso_1}
   (\forall i\in\{1,\ldots,N\})\quad
			\overline{p}_{i}=
			\begin{cases}
				0, & \text{if}\,\, i\in A_{k};\\
				\dfrac{1}{2\lambda}\left[\dfrac{(N-k+2\lambda)\sqrt{\alpha_{i}}}{\sum_{j\notin A_{k}}\sqrt{\alpha_{j}}} -1\right],& \text{if}\,\, i\notin A_{k}
			\end{cases}
		\end{align}
		is the solution to problem \eqref{prob_aux_2}.
	\end{proposition}
	\begin{proof}
		First of all, the set in \eqref{def_k} is nonempty since $N-1$ is in that set. Indeed, $(1+2\lambda)\sqrt{\alpha_{\ell_{N}}}>\sqrt{\alpha_{\ell_{N}}}$. Hence $k$ is well defined. Note that $\nabla \ell(p)=\left(\dfrac{\alpha_{i}}{(1+2\lambda p_{i})^{2}}\right)_{i=1}^{N}$ for all $p\in\R^{N}$. Then, since $-\ell$ is convex on $\Delta_{N}$, by the definition of $\Delta_{N}$ and the KKT's conditions, it follows that it is enough to prove that there exists $\tau\in \R$ and $(\mu_{i})_{i=1}^{N}\in\R_{+}^{N}$ such that
		\begin{align}
			\label{cond_kkt}
			&\dfrac{-\alpha_{i}}{(1+2\lambda p_{i})^{2}} +\tau-\mu_{i}=0,\quad\mu_{i}p_{i}=0,\quad p_{i}\geq 0\quad\text{for all }i\in\{1,\ldots,N\},\\
			&\text{and }\sum_{i=1}^{N}p_{i}=1,\label{cond_kkt_2}
		\end{align}
		where $p\in\R^{N}$ is defined by \eqref{sol_prob_paso_1}. Consider
		\begin{align*}
			\tau=\frac{1}{(N-k+2\lambda)^{2}}\left(\sum_{j\notin A_{k}}\sqrt{\alpha_{j}}\right)^{2}\in\R
		\end{align*}
		and $(\mu_{i})_{i=1}^{N}\in \R^{N}$ defined by $\mu_{i}=\tau-\alpha_{i}$ if $i\in A_{k}$ and $\mu_{i}=0$ if $i\notin A_{k}$. Thus, we have the second condition in \eqref{cond_kkt}. Let's now prove the first condition in \eqref{cond_kkt}. Let $i\in\{1,\ldots,N\}$. If $i\in A_{k}$, then $p_{i}=0$ and $\dfrac{-\alpha_{i}}{(1+2\lambda p_{i})^{2}}+\tau-\mu_{i}=-\alpha_{i}+\tau-\mu_{i}=0$. If $i\notin A_{k}$, then by \eqref{sol_prob_paso_1}, we have that $(1+2\lambda p_{i})^{2}=\dfrac{(N-k+2\lambda)^{2}\alpha_{i}}{\left(\sum_{j\notin A_{k}}\sqrt{\alpha_{j}}\right)^{2}}$ and hence
		\begin{align*}
			\frac{-\alpha_{i}}{(1+2\lambda p_{i})^{2}}+\tau-\mu_{i}=-\frac{\left(\sum_{j\notin A_{k}}\sqrt{\alpha_{j}}\right)^{2}}{(N-k+2\lambda)^{2}}+\tau-\mu_{i}=-\mu_{i}=0,
		\end{align*}
		which prove the first condition in \eqref{cond_kkt}.
		We claim that $\mu_{i}\geq 0$ for all $i\in A_{k}$ (note that if $k=0$, the latter is direct since $A_{0}=\emptyset$, so in order to prove this claim we assume that $k>0$). Let $i\in A_{k}$. Then $\alpha_{i}\leq \alpha_{\ell_{k}}$. Now, by definition of $k$, we have that $k-1$ is not in the set in \eqref{def_k}, that is,
		\begin{align*}
			(N-k+1+2\lambda)\sqrt{\alpha_{\ell_{k}}}\leq \sum_{j\notin A_{k-1}}\sqrt{\alpha_{j}} = \sqrt{\alpha_{\ell_{k}}}+\cdots+\sqrt{\alpha_{\ell_{N}}},
		\end{align*}
		which yields that
		\begin{align*}
			(N-k+2\lambda)\sqrt{\alpha_{i}}\leq (N-k+2\lambda)\sqrt{\alpha_{\ell_{k}}}\leq \sqrt{\alpha_{\ell_{k+1}}}+\cdots+\sqrt{\alpha_{\ell_{N}}}=\sum_{j\notin A_{k}}\sqrt{\alpha_{j}},
		\end{align*}
		which implies that $\mu_{i}=\tau-\alpha_{i}=\frac{1}{(N-k+2\lambda)^{2}}\left(\sum_{j\notin A_{k}}\sqrt{\alpha_{j}}\right)^{2}-\alpha_{i}\geq 0$. We claim now that $p_{i}\geq 0$ for all $i\notin A_{k}$. Let $i\notin A_{k}$. Then $\alpha_{i}\geq \alpha_{\ell_{k+1}}$ and therefore
		\begin{align*}
			\dfrac{(N-k+2\lambda)\sqrt{\alpha_{i}}}{\sum_{j\notin A_{k}}\sqrt{\alpha_{j}}}\geq \dfrac{(N-k+2\lambda)\sqrt{\alpha_{\ell_{k+1}}}}{\sum_{j\notin A_{k}}\sqrt{\alpha_{j}}}>1,
		\end{align*}
		where the last inequality is by the definition of $k$. Thus, by definition in \eqref{sol_prob_paso_1}, we obtain that $p_{i}\geq 0$. Finally, since $|A_{k}^{c}|=N-k$, we deduce from definition in \eqref{sol_prob_paso_1}, that $\sum_{i=1}^{N}p_{i}=1$. In summary, we have proved \eqref{cond_kkt}-\eqref{cond_kkt_2}.
	\end{proof}
Consider now the following function
\begin{align}
\label{sup_fun}
x\in\H\mapsto \widetilde{f}(x)=\sup_{p\in\mathcal{P}}\sum_{i=1}^{N}p_{i}(\langle x, Qx\rangle+\langle b_{i},x\rangle+c_{i}),
\end{align}
where $\mathcal{P}\subset\Delta_{N}$ is a nonempty closed convex set, $Q$ is a bounded linear operator which is symmetric positive semidefinite, $b_{i}\in\H$, and $c_{i}\in\R$. The following result provide an algorithm for compute the proximity operator of the function $\widetilde{f}$ in \eqref{sup_fun}.
\begin{proposition}
\label{calc_prox_sup_1}
Let $x\in\H$ and $\lambda>0$. Let $C=(\Id+2\lambda Q)^{-1}$ and $B\in\mathcal{L}(\R^{n},\H)$ defined by $Bp=\sum_{j=1}^{N}p_{j}b_{j}$. Let $M=\lambda B^{*}(\Id-\lambda C^{*}Q-\frac{1}{2}C^{*})CB\in\mathcal{L}(\R^{N})$ and $\gamma\in\R^{N}$ defined by
\begin{align*}
\gamma_{i}=\langle b_{i},C(2\lambda QC-2\Id+C)x\rangle-c_{i}.
\end{align*}
Then the following hold.
\begin{enumerate}
    \item $M$ is a symmetric positive semidefinite operator.
    \item Let $p^{0}\in\R^{N}$, $\overline{p}^{0}\in \R^{N}$, and $q^{0}\in\R^{N}$ such that $p^{0}=\overline{p}^{0}$. Let $L=M^{1/2}$ and let $\tau,\sigma>0$ such that $\tau\sigma\|L\|^{2}<1$. For every $k\in\mathbb{N}$, we consider the following routine
\begin{equation}
\label{alg_prob_aux_3}
\left\lfloor
\begin{array}{ll}
x^{k}=C(x-\lambda Bp^{k})\\
q^{k+1}=\frac{2}{\sigma+2}(q^{k}+\sigma L\overline{p}^{k})\\
u^{k+1}=p^{k}-\tau(L^{*}q^{k+1}+\gamma)\\
p^{k+1}=\Proj_{\mathcal{P}}(u^{k+1})\\
\overline{p}^{k+1}=p^{k+1}+u^{k+1}-p^{k}.
\end{array}
\right.
\end{equation}
Then $x^{k}\rightarrow P_{\lambda}\widetilde{f}(x)$.
\end{enumerate}
\end{proposition}
\begin{proof}
For every $p\in\mathcal{P}$, let us define the function
\begin{align*}
x\in\H\mapsto \widetilde{g}_{p}(x)=\sum_{i=1}^{N}p_{i}(\langle x,Qx\rangle+\langle b_{i},x\rangle+c_{i}).
\end{align*}
Note that $\{\widetilde{g}_{p}\}_{p\in\mathcal{P}}\subset\Gamma_{0}(\H)$. Moreover, the function $p\mapsto \widetilde{g}_{p}(x)$ is concave and upper semicontinuous for all $x\in\H$. In addition, since $\mathcal{P}\subset\Delta_{N}$, then $\widetilde{f}=\sup_{p\in\mathcal{P}}\widetilde{g}_{p}$ is proper. Furthermore, $\mathcal{P}$ is a nonempty compact and convex set.
Then, by \cite[Theorem~3.5]{MR4279933}, we have that 
  \begin{align}
  \label{prox_sup}
P_{\lambda}\widetilde{f}(x)=P_{\lambda}\widetilde{g}_{\overline{p}}(x),\quad\text{ with } \overline{p}\in \displaystyle\argmax_{p\in\mathcal{P}}e_{\lambda}\widetilde{g}_{p}(x).
  \end{align}
Note that $\widetilde{g}_{p}$ is differentiable. Then, given $\overline{p}\in\mathcal{P}\subset\Delta_{N}$, we have
\begin{align}
y=P_{\lambda}\widetilde{g}_{\overline{p}}(x)&\Leftrightarrow x=y+\lambda\nabla \widetilde{g}_{\overline{p}}(y)\nonumber\\
&\Leftrightarrow x=y+\lambda\sum_{i=1}^{N}\overline{p}_{i}(2Qy+b_{i})\nonumber\\
&\Leftrightarrow x=(\Id+2\lambda Q)y+\lambda\sum_{i=1}^{N}\overline{p}_{i}b_{i}\nonumber\\
&\Leftrightarrow y=(\Id+2\lambda Q)^{-1}\left(x-\lambda\sum_{i=1}^{N}\overline{p}_{i}b_{i}\right)\nonumber\\
&\Leftrightarrow y=C(x-\lambda B\overline{p}),\label{prox_fp_1}
\end{align}
where $C=(\Id+2\lambda Q)^{-1}$ and $B$ is the linear operator defined by $p\mapsto Bp=\sum_{i=1}^{N}p_{i}b_{i}$. Let us compute $e_{\lambda}\widetilde{g}_{p}(x)$ with $p\in\mathcal{P}\subset\Delta_{N}$. By \eqref{prox_fp_1}, we obtain that
 \begin{align}
e_{\lambda}\widetilde{g}_{p}(x)&=\widetilde{g}_{p}(P_{\lambda}\widetilde{g}_{p}(x))+\frac{1}{2\lambda}\|x-P_{\lambda}\widetilde{g}_{p}(x)\|^{2}\nonumber\\
&=\sum_{i=1}^{N}p_{i}(\langle C(x-\lambda Bp),QC(x-\lambda Bp)\rangle+\langle b_{i},C(x-\lambda Bp)\rangle+c_{i})+\dfrac{1}{2\lambda}\|x-C(x-\lambda Bp)\|^{2}\label{mor_env_fp_1}\\
&=:\ell_{1}(p,x)+\ell_{2}(p,x).\nonumber
 \end{align}
  Let $c=(c_{i})_{i=1}^{N}$. Then the first term in \eqref{mor_env_fp_1} is
 \begin{align}
\ell_{1}(p,x)&=\langle C(x-\lambda Bp),QC(x-\lambda Bp)\rangle+\langle Bp,C(x-\lambda Bp)\rangle+p^{\top}c\nonumber\\
&=\langle Cx,QCx\rangle-\lambda\langle Cx,QCBp\rangle-\lambda\langle CBp,QCx\rangle+\lambda^{2}\langle CBp,QCBp\rangle\nonumber\\
&+\langle Bp,Cx\rangle-\lambda\langle Bp,CBp\rangle+p^{\top}c\nonumber\\
&=\langle Cx,QCx\rangle-2\lambda\langle Cx,QCBp\rangle+\lambda^{2}\langle p,(CB)^{*}QCBp\rangle\nonumber\\
&+\langle Bp,Cx\rangle-\lambda\langle p,B^{*}CBp\rangle+p^{\top}c\nonumber\\
&=\langle Cx,QCx\rangle+\lambda^{2}\langle p,(CB)^{*}QCBp\rangle-\lambda\langle p,B^{*}CBp\rangle\nonumber\\
&+\langle Bp,Cx-2\lambda CQCx\rangle+p^{\top}c.\nonumber
 \end{align}
On the other hand, the second term in \eqref{mor_env_fp_1} is
\begin{align}
\ell_{2}(p,x)&=\dfrac{1}{2\lambda}\|(\Id-C)x+\lambda CBp\|^{2}\nonumber\\
&=\frac{1}{2\lambda}\|(\Id-C)x\|^{2}+\langle (\Id-C)x,CBp \rangle+\frac{\lambda}{2}\|CBp\|^{2}.\nonumber\\
&=\dfrac{1}{2\lambda}\|(\Id-C)x\|^{2}+\langle Bp,C(\Id-C)x\rangle+\frac{\lambda}{2}\langle p,(CB)^{*}CBp\rangle.\nonumber
\end{align}
Let $\gamma\in\R^{N}$ defined by
\begin{align}
\gamma_{i}&=\langle b_{i},2\lambda CQCx-Cx-C(\Id-C)x\rangle-c_{i}=\langle b_{i},C(2\lambda QC-2\Id+C) x\rangle-c_{i}.\nonumber
\end{align}
Then
\begin{align}
e_{\lambda}\widetilde{g}_{p}(x)=\frac{1}{2\lambda}\|(\Id-C)x\|^{2}+\langle Cx,QCx\rangle-p^{\top}\gamma-\langle p,Mp\rangle,\nonumber
\end{align}
where $M=\lambda B^{*}CB-\lambda^{2}(CB)^{*}QCB-\frac{\lambda}{2}(CB)^{*}CB=\lambda B^{*}(\Id-\lambda C^{*}Q-\frac{1}{2}C^{*})CB$.
Therefore the problem $\max_{p\in\mathcal{P}}e_{\lambda}\widetilde{g}_{p}(x)$ reduces to
\begin{align}
\label{prob_aux_3}
\max_{p\in\mathcal{P}}-p^{\top}\gamma-\langle p,Mp\rangle\Leftrightarrow \min_{p\in\mathcal{P}}p^{\top}\gamma+\langle p,Mp\rangle.
\end{align}
Observe that, since $C$ and $Q$ are symmetric, $M$ also is symmetric. Now, note that $p\mapsto e_{\lambda}\widetilde{g}_{p}(x)$ is concave since is the infimum of affine linear functions. Then, $p\mapsto\langle p,Mp\rangle+p^{\top}\gamma$ is convex. Thus, $M$ is positive semidefinite. Hence there exists $M^{1/2}$ symmetric positive semidefinite such that $M=M^{1/2}M^{1/2}$. Then, $\langle p,Mp\rangle=\|M^{1/2}p\|^{2}$. Thus, the problem~\eqref{prob_aux_3} is equivalent to
\begin{align}
\label{prob_aux_3_2}
\min_{p\in\mathcal{P}}\overline{g}(Lp)+\overline{h}(p),
\end{align}
where $\overline{g}=\|\cdot\|^{2}$, $L=M^{1/2}$, and $\overline{h}$ is defined by $\overline{h}(p)=p^{\top}\gamma$. Now, we have that $P_{\sigma}\overline{g}(x)=\frac{1}{1+2\sigma}x$, which implies that $P_{\sigma}\overline{g}^{*}(x)=x-\sigma P_{1/\sigma}\overline{g}(x/\sigma)=x-\frac{\sigma}{1+2/\sigma}\cdot\frac{x}{\sigma}=\frac{2}{\sigma+2}x$. In addition, $\nabla \overline{h}(p)=\gamma$ is Lipschitz. Thus, the algorithm \eqref{alg_prob_aux_3} is equivalent to
\begin{equation*}
\left\lfloor
\begin{array}{ll}
x^{k}=C(x-\lambda Bp^{k})\\
q^{k+1}=P_{\sigma}\overline{g}^{*}(q^{k}+\sigma L\overline{p}^{k})\\
u^{k+1}=p^{k}-\tau(L^{*}q^{k+1}+\nabla\overline{h}(p^{k}))\\
p^{k+1}=\Proj_{\mathcal{P}}(u^{k+1})\\
\overline{p}^{k+1}=p^{k+1}+u^{k+1}-p^{k}.
\end{array}
\right.
\end{equation*}
From \cite[Theorem~3.1]{briceno2019projected}, there exists $\overline{p}\in\mathcal{P}$ solution to \eqref{prob_aux_3_2} such that $p^{k}\rightarrow\overline{p}$. That is $\overline{p}\in\argmax_{p\in\mathcal{P}}e_{\lambda}\widetilde{g}_{p}(x)$. Now, by continuity, $x^{k}\rightarrow C(x-\lambda B\overline{p})$. Finally, by \eqref{prox_sup} and \eqref{prox_fp_1}, $C(x-\lambda B\overline{p})=P_{\lambda}\widetilde{f}(x)$.
\end{proof}
\begin{remark}
In the case when $\mathcal{P}=\Delta_{N}$, the projection onto $\mathcal{P}$ in \eqref{alg_prob_aux_3} has an explicit form given in \cite{wang2013projection}. In the case when $\mathcal{P}$ is the ambiguity set given in \ref{item1_2}, we have that $\mathcal{P}=\Delta_{N}\cap\mathcal{P}_{2}$ and the projection onto $\mathcal{P}$ can be calculated using the Dykstra's algorithm \cite[Proposition~5.3]{combettes2011proximal}, which requieres compute the projection onto $\Delta_{N}$ and $\mathcal{P}_{2}$. Note that by using \cite[Proposition~29.3]{MR3616647} the projection onto $\mathcal{P}_{2}$ has an explicit form. Similarly, if $\mathcal{P}$ is the set in \ref{item1_3}, then $\mathcal{P}=\Delta_{N}\cap\mathcal{P}_{3}$ and we can use the Dykstra's algorithm for computing the projection onto $\mathcal{P}$. Note that the projection onto $\mathcal{P}_{3}$ has an explicit form given in \cite[Example~29.21]{MR3616647}.
\end{remark}
 \section{Algorithms}
 \label{sec_alg}
 We will see a splitting algorithm for solving problems \eqref{prob_item2} and \eqref{prob_item3}. This method is based on the results of Section~\ref{sec_prox_sup}, in which we compute the proximity operator of the supremum function appearing in problem~\eqref{gen_dro_disc}. Note that the difficulty of calculating this proximity operator depends on the functions $f_{i}$ and the set $\mathcal{P}$.

On the other hand, to solve problem~\eqref{prob_item4}, we propose a second method, which focuses in the resolution of an equivalent monotone inclusion problem, coming from the optimality conditions via the computation of resolvents of maximally monotone operators involved given in
\cite[Proposition~4.1]{briceno2023perturbation}. In addition, we propose a third and fourth method to solve the problem~\eqref{prob_item4}, which correspond to the application of the algorithms proposed in \cite{briceno2015forward} and \cite{davis2017three}, respectively, to solve a product space formulation involving the normal cone to a closed vector subspace.

\subsection{A proximal algorithm}
Note that the problem~\eqref{gen_dro_disc} is equivalent to
 \begin{align}
 \label{equiv_dro_disc}
\min_{\mathbf{x}\in V} f(\mathbf{x})+\iota_{\overline{\Q}}(\mathbf{x})+H(\mathbf{x}),
 \end{align}
 where $\overline{\Q}=\bigtimes_{i=1}^{N}\Q_{i}$ and $f$ is the function defined in \eqref{fun_sup_sep_1}. The problem~\eqref{equiv_dro_disc} can be solved using the algorithm proposed in \cite{briceno2023primal}, which requires calculating the proximity of $f$, the projection on $\overline{\Q}$, and $\nabla H$.

 Let us define the following functions
 \begin{align*}
		(\alpha,\beta)\in\R_{++}^{N}\times\R^{N}\mapsto \rho(\alpha,\beta)=\argmin_{p\in\mathcal{P}}\left\{\dfrac{1}{2}\sum_{i=1}^{N}p_{i}^{2}\alpha_{i}-\sum_{i=1}^{N}p_{i}\beta_{i}\right\},
	\end{align*}
 \begin{align*}
		(\alpha,\lambda)\in\R_{+}^{N}\times\R_{++}\mapsto \varphi(\alpha,\lambda)=\argmax_{p\in\Delta_{N}}\sum_{i=1}^{N}\left(\frac{p_{i}}{1+2\lambda p_{i}}\right)\alpha_{i}.
\end{align*}
	In order to solve the problem in the definition of $\rho$, we use the method proposed in \cite{cominetti2014newton} for the cases \ref{item1_1} and \ref{item1_2}, and the algorithm \eqref{dykstra_alg_1} for the case \ref{item1_3}. On the other hand, the explicit form of $\varphi$ is given by \eqref{sol_prob_paso_1}.
 From Proposition~\ref{calc_prox_sup}\ref{case_1}, \cite[Proposition~24.8(ix)]{MR3616647}, and \cite[Theorem~3.1]{briceno2023primal}, we obtain the following
	\begin{proposition}
		\label{conv_alg_aplic_2}
		Let $\mathbf{u}^{0}\in\H^{N}$, let $\mathbf{x}^{0}\in V$, let $\overline{\mathbf{x}}^{0}=\mathbf{x}^{0}$, and let $\mathbf{y}^{0}\in V^{\perp}$. Let $\lambda\in\left]0,2\beta\right[$ and let $\gamma>0$ such that $\gamma<\frac{1}{\lambda}-\frac{1}{2\beta}$. For every $k\in\mathbb{N}$, we consider the following routine.
		
\begin{equation}
\label{alg_pi_2}
\left\lfloor
\begin{array}{ll}
u_{i}^{k+1}=u_{i}^{k}+\gamma \overline{x}_{i}^{k}-\gamma \Proj_{\Q_{i}}\left( \frac{u_{i}^{k}}{\gamma}+\overline{x}_{i}^{k}\right)\quad\textnormal{for all }i\in \{1,\ldots,N\}\\
\overline{\mathbf{z}}^{k}=\mathbf{x}^{k}+\lambda \mathbf{y}^{k}-\lambda\Proj_{V}(\mathbf{u}^{k+1}+\nabla H(\mathbf{x}^{k}))\\
p^{k}=\rho((\lambda\|a_{i}\|^{2})_{i=1}^{N},(\langle a_{i},\overline{z}_{i}^{k}\rangle+\xi_{i})_{i=1}^{N})\\
	w_{i}^{k+1}=\overline{z}_{i}^{k}-\lambda p_{i}^{k}a_{i}\quad\textnormal{for all}\,\, i\in\{1,\ldots,N\}\\
	\mathbf{x}^{k+1}=\Proj_{V} \mathbf{w}^{k+1}\\
\mathbf{y}^{k+1}=\mathbf{y}^{k}+(\mathbf{x}^{k+1}-\mathbf{w}^{k+1})/\lambda\\
 \overline{\mathbf{x}}^{k+1}=2\mathbf{x}^{k+1}-\mathbf{x}^{k}.
   \end{array}
   \right.
\end{equation}
Then there exists $\mathbf{x}\in\H$ solution to problem~\eqref{prob_item2} such that $\mathbf{x}^{k}\rightharpoonup \mathbf{x}$.
	\end{proposition}
 Similarly, from Proposition~\ref{calc_prox_sup}\ref{case_2}, \cite[Proposition~24.8(ix)]{MR3616647}, and \cite[Theorem~3.1]{briceno2023primal}, we obtain the following
 \begin{proposition}
		\label{conv_alg_aplic_1}
		Let $\mathbf{u}^{0}\in\H^{N}$, let $\mathbf{x}^{0}\in V$, let $\overline{\mathbf{x}}^{0}=\mathbf{x}^{0}$, and let $\mathbf{y}^{0}\in V^{\perp}$. Let $\lambda\in\left]0,2\beta\right[$ and let $\gamma>0$ such that $\gamma<\frac{1}{\lambda}-\frac{1}{2\beta}$. For every $k\in\mathbb{N}$, we consider the following routine.
  \begin{equation}
\label{alg_pi_1}
   \left\lfloor
   \begin{array}{ll}
u_{i}^{k+1}=u_{i}^{k}+\gamma \overline{x}_{i}^{k}-\gamma \Proj_{\Q_{i}}\left( \frac{u_{i}^{k}}{\gamma}+\overline{x}_{i}^{k}\right)\quad\textnormal{for all }i\in \{1,\ldots,N\}\\
\overline{\mathbf{z}}^{k}=\mathbf{x}^{k}+\lambda \mathbf{y}^{k}-\lambda\Proj_{V}(\mathbf{u}^{k+1}+\nabla H(\mathbf{x}^{k}))\\
p^{k}=\varphi((\|\overline{z}_{i}^{k}-\xi_{i}\|^{2})_{i=1}^{N},\lambda)\\
\vspace{0.1cm}
	w_{i}^{k+1}=\dfrac{1}{1+2\lambda p_{i}^{k}}(\overline{z}_{i}^{k}+2\lambda p_{i}^{k}\xi_{i})\quad\textnormal{for all}\,\, i\in\{1,\ldots,N\}\\
\mathbf{x}^{k+1}=\Proj_{V} \mathbf{w}^{k+1}\\
\mathbf{y}^{k+1}=\mathbf{y}^{k}+(\mathbf{x}^{k+1}-\mathbf{w}^{k+1})/\lambda\\
\overline{\mathbf{x}}^{k+1}=2\mathbf{x}^{k+1}-\mathbf{x}^{k}.
   \end{array}
   \right.
		\end{equation}
		Then there exists $\mathbf{x}\in\H$ solution to problem~\eqref{prob_item3} such that $\mathbf{x}^{k}\rightharpoonup \mathbf{x}$.
	\end{proposition}

 
 \subsection{Distributed forward-backward method}
 In this section, we propose an alternative method for solving problem~\eqref{prob_principal}. Note that in the cases~\ref{item1} and \ref{item2}, we have that $\dom f_{i}=\H$ for all $i\in\{1,\ldots,N\}$. In addition, $\dom h=\H$. Thus, by the Fermat's rule and \cite[Corollary~16.50(iv)]{MR3616647}, the problem~\eqref{prob_principal} is equivalent to solve the following inclusion (optimality condition)
 \begin{align}
 \label{inc_1}
 \text{Find}\quad (x,p)\in\H\times\R^{N}\text{ such that }
\begin{cases}
0\in \nabla h(x)+N_{\Q}(x)+\sum_{i=1}^{N}p_{i}\partial f_{i}(x)\\
0\in N_{\mathcal{P}}(p)-G(x),
\end{cases}
 \end{align}
 where $G(x)=(f_{1}(x),\ldots, f_{N}(x))^{\top}\in\R^{N}$. If $\mathcal{P}=\Delta_{N}$, then $\mathcal{P}=\R_{+}^{N}\cap \mathcal{P}_{1}$, where $\mathcal{P}_{1}=\{p\in\R^{N}\,:\,\sum_{i=1}^{N}p_{i}=1\}$. Since $\text{int}(\R_{+}^{N})\cap \mathcal{P}_{1}\neq\emptyset$, then by \cite[Corollary~16.48(ii)]{MR3616647}, $N_{\mathcal{P}}=N_{\R_{+}^{N}}+N_{\mathcal{P}_{1}}$. Now, by defining $\phi=\iota_{\R_{-}}$, we have that $N_{\R_{+}}=\partial\phi^{*}$. Define 
 \begin{align}
&(x,p)\in\H\times\R^{N}\mapsto B_{i}(x,p):=\begin{pmatrix}
p_{i}\partial f_{i}(x)\\ \partial\phi^{*}(p_{i})e_{i}-f_{i}(x)e_{i}
\end{pmatrix}\quad\text{for all }i\in\{1,\ldots,N\}\label{B_i}\\
&(x,p)\in\H\times\R^{N}\mapsto A_{1}(x,p)=\begin{pmatrix}
N_{\Q}(x)\\ N_{\mathcal{P}_{1}}(p)
\end{pmatrix}\nonumber\\
&(x,p)\in\H\times\R^{N}\mapsto C(x,p)=\begin{pmatrix}
\nabla h(x)\\ 0\nonumber
\end{pmatrix},
 \end{align}
 where $e_{i}$ is the canonical vector in $\R^{N}$. Then the inclusion \eqref{inc_1} is equivalent to
 \begin{align}
\label{inc_2}
\begin{pmatrix}
0\\0
\end{pmatrix}
\in A_{1}(x,p)+\sum_{i=1}^{N}B_{i}(x,p)+\sum_{i=1}^{N}C_{i}(x,p),
 \end{align}
 where $C_{1}=C$ and $C_{i}=0$ for all $i\in\{2,\ldots,N\}$. If $\mathcal{P}=\{p\in\Delta_{N}\,:\,p\leq q\}$, then $\mathcal{P}=\R_{+}^{N}\cap\mathcal{P}_{1}\cap\mathcal{P}_{2}$, where $\mathcal{P}_{2}=\{p\in\R^{N}\,:\,p\leq q\}$. By assumption in \ref{item1_2}, $\mathcal{P}_{1}\cap\text{int}(\R_{+}^{N})\cap\text{int}(\mathcal{P}_{2})\neq\emptyset$. Hence, by \cite[Corollary~16.50(iv)]{MR3616647}, $N_{\mathcal{P}}=N_{\R_{+}^{N}}+N_{\mathcal{P}_{1}}+N_{\mathcal{P}_{2}}$. Thus, in this case the inclusion \eqref{inc_1} is equivalent to
  \begin{align}
\label{inc_2_2}
\begin{pmatrix}
0\\0
\end{pmatrix}
\in A_{1}(x,p)+A_{2}(x,p)+\sum_{i=1}^{N}B_{i}(x,p)+\sum_{i=1}^{N+1}C_{i}(x,p),
 \end{align}
 where $(x,p)\in\H\times\R^{N}\mapsto A_{2}(x,p)=\begin{pmatrix} 0\\N_{\mathcal{P}_{2}}(p)\end{pmatrix}$ and $C_{N+1}=0$. Similarly, if $\mathcal{P}=\{p\in\Delta_{N}\,:\,\mu_{-}\leq\sum_{i=1}^{N}p_{i}\xi_{i}\leq\mu_{+}\}$, we deduce that the inclusion \eqref{inc_1} is equivalent to
  \begin{align}
\label{inc_2_3}
\begin{pmatrix}
0\\0
\end{pmatrix}
\in A_{1}(x,p)+A_{3}(x,p)+\sum_{i=1}^{N}B_{i}(x,p)+\sum_{i=1}^{N+1}C_{i}(x,p),
 \end{align}
 where $(x,p)\in\H\times\R^{N}\mapsto A_{3}(x,p)=\begin{pmatrix} 0\\N_{\mathcal{P}_{3}}(p)\end{pmatrix}$ and $\mathcal{P}_{3}=\{p\in\R^{N}\,:\,\mu_{-}\leq \sum_{i=1}^{N}p_{i}\xi_{i}\leq\mu_{+}\}$.
 For apply the splitting methods in order to solve the previous inclusions, we need to prove that $B_{i}$ is maximally monotone and compute $J_{\gamma B_{i}}$ with $\gamma>0$. 
 \begin{proposition}
\label{res_B_i}
Let $i\in\{1,\ldots,N\}$ and consider the operator $B_{i}$ defined in \eqref{B_i}, where $f_{i}\in\Gamma_{0}(\H)$ satisfies $\dom f_{i}=\H$. Then the following hold.
\begin{enumerate}
    \item\label{prop_item1} $B_{i}$ is maximally monotone.
    \item\label{prop_item2} Let $(x,p)\in\H\times\R^{N}$, let $\gamma>0$, and let $\overline{\omega}$ the unique real number in $\left[0,+\infty\right[$ such that
    \begin{align}
 \label{def_omega_2}
\overline{\omega}=\begin{cases}
0 &\text{ if } p_{i}+\gamma f_{i}(x)\leq 0\\
p_{i}+\gamma f_{i}(P_{\gamma\overline{\omega}}f_{i}(x)) &\text{ if }p_{i}+\gamma f_{i}(x)>0.
\end{cases}
 \end{align}
 Then
 \begin{align*}
J_{\gamma B_{i}}(x,p)=\begin{cases}
(x,\omega)&\text{ if } p_{i}+\gamma f_{i}(x)\leq 0\\
(P_{\gamma\omega_{i}}f_{i}(x),\omega)&\text{ if }p_{i}+\gamma f_{i}(x)>0,
\end{cases}
 \end{align*}
 where $\omega_{i}=\overline{\omega}$ and $\omega_{j}=p_{j}$ for all $j\neq i$.
\end{enumerate}
 \end{proposition}
 \begin{proof}
\ref{prop_item1}: First, we prove that $B_{i}$ is monotone. Let $(u,v)\in B_{i}(x,p)$ and $(u',v')\in B_{i}(x',p')$. Then $u\in p_{i}\partial f_{i}(x)$, $u'\in p_{i}'\partial f_{i}(x')$, $v_{i}\in \partial\phi^{*}(p_{i})-f_{i}(x)$, and $v_{i}'\in \partial\phi^{*}(p_{i}')-f_{i}(x')$. In addition, $v_{j}=v_{j}'=0$ for all $j\neq i$. By \cite[Proposition~3.1(v)]{briceno2023perturbation}, we have that the following operator
\begin{align}
\label{op_kkt}
(x,\xi)\in\H\times\R\mapsto \xi\partial f_{i}(x)\times (\partial \phi^{*}(\xi)-f_{i}(x)) 
\end{align}
is maximally monotone. Thus,
\begin{align*}
\langle(u,v)-(u',v'),(x,p)-(x',p')\rangle=\langle u-u',x-x'\rangle+\langle v_{i}-v_{i}',p_{i}-p_{i}'\rangle\geq 0, 
\end{align*}
which proves that $B_{i}$ is monotone. By \cite[Theorem~21.1]{MR3616647}, it is enough to prove that $\ran (\Id+B_{i})=\H\times\R^{N}$. Let $(y,\omega)\in\H\times\R^{N}$. Then $(y,\omega_{i})\in\H\times\R$. Since the operator \eqref{op_kkt} is maximally monotone, then from \cite[Theorem~21.1]{MR3616647}, there exists $(x,\xi)\in\H\times\R$ such that
\begin{align*}
(y,\omega_{i})\in (x,\xi)+\begin{pmatrix}
\xi\partial f_{i}(x)\\ \partial\phi^{*}(\xi)-f_{i}(x)
\end{pmatrix}.
\end{align*}
Let $p\in\R^{N}$ defined by $p_{i}=\xi$ and $p_{j}=\omega_{j}$ for all $j\neq i$. Then
\begin{align*}
(y,\omega)\in (x,p)+\begin{pmatrix}
p_{i}\partial f_{i}(x)\\ \partial\phi^{*}(p_{i})e_{i}-f_{i}(x)e_{i}
\end{pmatrix}=(x,p)+B_{i}(x,p).
\end{align*}
That is, $\ran (\Id+B_{i})=\H\times\R^{N}$.
\ref{prop_item2}: Let $(y,\omega)\in\H\times \R^{N}$. Then
 \begin{align*}
(y,\omega)=J_{\gamma B_{i}}(x,p)&\quad\Leftrightarrow\quad (x,p)\in (y,\omega)+\gamma B_{i}(y,\omega)\\
&\quad\Leftrightarrow\quad \begin{cases}
x\in y+\gamma \omega_{i}\partial f_{i}(y)\\
p_{i}\in\omega_{i}+\gamma\partial\phi^{*}(\omega_{i})-\gamma f_{i}(y)\\
p_{j}=\omega_{j}\text{ for all }j\neq i
\end{cases}\\
&\quad\Leftrightarrow\quad\begin{cases}
y=P_{\gamma\omega_{i}}f_{i}(x)\\
\omega_{i}=P_{\gamma}\phi^{*}(p_{i}+\gamma f_{i}(y))\\
\omega_{j}=p_{j}\text{ for all }j\neq i
\end{cases}
 \end{align*}
 Using \cite[Remark~4.1 \& Example~4.1]{briceno2023perturbation} and that $\dom f_{i}=\H$, we deduce that
 \begin{align}
 \label{def_omega}
\omega_{i}=\begin{cases}
0 &\text{ if } p_{i}+\gamma f_{i}(x)\leq 0\\
p_{i}+\gamma f_{i}(P_{\gamma\omega_{i}}f_{i}(x)) &\text{ if }p_{i}+\gamma f_{i}(x)>0.
\end{cases}
 \end{align}
 Therefore,
 \begin{align*}
J_{\gamma B_{i}}(x,p)=\begin{cases}
(x,\omega)&\text{ if } p_{i}+\gamma f_{i}(x)\leq 0\\
(P_{\gamma\omega_{i}}f_{i}(x),\omega)&\text{ if }p_{i}+\gamma f_{i}(x)>0,
\end{cases}
 \end{align*}
 where $\omega_{i}\in\left[0,+\infty\right[$ is the unique solution to \eqref{def_omega} and $\omega_{j}=p_{j}$ for all $j\neq i$.
 \end{proof}
  In the case~\ref{item1}, $f_{i}(x)=\langle a_{i},x\rangle+\xi_{i}$. Then, $P_{\gamma \omega} f_{i}(x)=x-\gamma\omega a_{i}$ and thus $f_{i}(P_{\gamma \omega} f_{i}(x))=\langle a_{i},x-\gamma\omega a_{i}\rangle+\xi_{i}$. Hence
 \begin{align}
\omega=p_{i}+\gamma f_{i}(P_{\gamma\omega}f_{i}(x))&\Leftrightarrow \omega=p_{i}+\gamma\langle a_{i},x-\gamma\omega a_{i}\rangle+\gamma\xi_{i}\nonumber\\
&\Leftrightarrow \omega=\dfrac{p_{i}+\gamma(\langle a_{i},x\rangle+\xi_{i})}{1+\gamma^{2}\|a_{i}\|^{2}}\label{omega_1}.
 \end{align}
 It follows from \eqref{omega_1} that, in the case~\ref{item1}, the solution to \eqref{def_omega_2} reduces to
 \begin{align*}
\overline{\omega}=\begin{cases}
0 &\text{ if } p_{i}+\gamma(\langle a_{i},x\rangle+\xi_{i})\leq 0\\
\dfrac{p_{i}+\gamma(\langle a_{i},x\rangle+\xi_{i})}{1+\gamma^{2}\|a_{i}\|^{2}} &\text{ if }p_{i}+\gamma(\langle a_{i},x\rangle+\xi_{i})>0.
\end{cases}
 \end{align*}
 On the other hand, $J_{\gamma A_{1}}(x,p)=(\Proj_{\Q}(x),\Proj_{\mathcal{P}_{1}}(p))$, $J_{\gamma A_{2}}(x,p)=(x,\Proj_{\mathcal{P}_{2}}(p))$, and $J_{\gamma A_{3}}(x,p)=(x,\Proj_{\mathcal{P}_{3}}(p))$. Once calculated $J_{\gamma A_{1}}$, $J_{\gamma A_{2}}$, $J_{\gamma A_{3}}$, and $J_{\gamma B_{i}}$ for all $i\in\{1,\ldots,N\}$, we can apply the method proposed in \cite{aragon2023distributed} for solving problems~\eqref{inc_2}, \eqref{inc_2_2}, and \eqref{inc_2_3}. Define $\mathbf{H}=\mathcal{H}\times\R^{N}$. From \cite[Theorem~3]{aragon2023distributed}, we deduce the following results.
\begin{proposition}
\label{prop_aragon_1}
Let $\lambda\in\left]0,2\beta\right[$, let $\gamma\in\left]0,1-\frac{\lambda}{2\beta}\right[$, and let $\overline{\mathbf{z}}^{0}=((\overline{x}_{1}^{0},\overline{p}_{1}^{0}),\ldots,(\overline{x}_{N}^{0},\overline{p}_{N}^{0}))\in\mathbf{H}^{N}$. For every $k\in\mathbb{N}$, we consider the following routine.
\begin{equation}
\label{alg_aragon_1}
\left\lfloor
\begin{array}{ll}
  x_{1}^{k}:=\Proj_{\Q}(\overline{x}_{1}^{k})\textnormal{ and }p_{1}^{k}:=\Proj_{\mathcal{P}_{1}}(\overline{p}_{1}^{k})\\
(x_{2}^{k},p_{2}^{k}):=J_{\lambda B_{1}}(\overline{x}_{2}^{k}+x_{1}^{k}-\overline{x}_{1}^{k}-\lambda\nabla h(x_{1}^{k}),\overline{p}_{2}^{k}+p_{1}^{k}-\overline{p}_{1}^{k})\\
(x_{i}^{k},p_{i}^{k}):=J_{\lambda B_{i-1}}(\overline{x}_{i}^{k}+x_{i-1}^{k}-\overline{x}_{i-1}^{k},\overline{p}_{i}^{k}+p_{i-1}^{k}-\overline{p}_{i-1}^{k})\textnormal{ for all }i\in\{3,\ldots,N\}\\
(x_{N+1}^{k},p_{N+1}^{k}):=J_{\lambda B_{N}}(x_{1}^{k}+x_{N}^{k}-\overline{x}_{N}^{k},p_{1}^{k}+p_{N}^{k}-\overline{p}_{N}^{k})\\
\left.
\begin{array}{ll}
\overline{x}_{j}^{k+1}=\overline{x}_{j}^{k}+\gamma(x_{j+1}^{k}-x_{j}^{k})\\
\overline{p}_{j}^{k+1}=\overline{p}_{j}^{k}+\gamma(p_{j+1}^{k}-p_{j}^{k})
\end{array}
\right\}
\textnormal{ for all }j\in\{1,\ldots,N\}
\end{array}
\right.
\end{equation}
Then there exists $(x,p)\in\mathbf{H}$ solution to problem~\eqref{inc_2} such that $(x_{1}^{k},p_{1}^{k})\rightharpoonup (x,p)$.
\end{proposition}

\begin{proposition}
\label{prop_aragon_2}
Let $\lambda\in\left]0,2\beta\right[$, $\gamma\in\left]0,1-\frac{\lambda}{2\beta}\right[$, and $\overline{\mathbf{z}}^{0}=((\overline{x}_{1}^{0},\overline{p}_{1}^{0}),\ldots,(\overline{x}_{N+1}^{0},\overline{p}_{N+1}^{0}))\in\mathbf{H}^{N+1}$. For every $k\in\mathbb{N}$, we consider the following routine.
\begin{equation}
\label{alg_aragon_2}
\left\lfloor
\begin{array}{ll}
  x_{1}^{k}:=\Proj_{\Q}(\overline{x}_{1}^{k})\textnormal{ and }p_{1}^{k}:=\Proj_{\mathcal{P}_{1}}(\overline{p}_{1}^{k})\\
(x_{2}^{k},p_{2}^{k}):=(\overline{x}_{2}^{k}+x_{1}^{k}-\overline{x}_{1}^{k}-\lambda\nabla h(x_{1}^{k}),\Proj_{\mathcal{P}}(\overline{p}_{2}^{k}+p_{1}^{k}-\overline{p}_{1}^{k}))\\
(x_{i}^{k},p_{i}^{k}):=J_{\lambda B_{i-2}}(\overline{x}_{i}^{k}+x_{i-1}^{k}-\overline{x}_{i-1}^{k},\overline{p}_{i}^{k}+p_{i-1}^{k}-\overline{p}_{i-1}^{k})\textnormal{ for all }i\in\{3,\ldots,N+1\}\\
(x_{N+2}^{k},p_{N+2}^{k}):=J_{\lambda B_{N}}(x_{1}^{k}+x_{N+1}^{k}-\overline{x}_{N+1}^{k},p_{1}^{k}+p_{N+1}^{k}-\overline{p}_{N+1}^{k})\\
\left.
\begin{array}{ll}
\overline{x}_{j}^{k+1}=\overline{x}_{j}^{k}+\gamma(x_{j+1}^{k}-x_{j}^{k})\\
\overline{p}_{j}^{k+1}=\overline{p}_{j}^{k}+\gamma(p_{j+1}^{k}-p_{j}^{k})
\end{array}
\right\}
\textnormal{ for all }j\in\{1,\ldots,N+1\}
\end{array}
\right.
\end{equation}
Then the following hold.
\begin{enumerate}
    \item If $\mathcal{P}=\mathcal{P}_{2}$, then there exists $(x,p)\in\mathbf{H}$ solution to problem~\eqref{inc_2_2} such that $(x_{1}^{k},p_{1}^{k})\rightharpoonup (x,p)$.
    \item If $\mathcal{P}=\mathcal{P}_{3}$, then there exists $(x,p)\in\mathbf{H}$ solution to problem~\eqref{inc_2_3} such that $(x_{1}^{k},p_{1}^{k})\rightharpoonup (x,p)$.
\end{enumerate}
\end{proposition}
\begin{remark}
By \cite[Example~29.18]{MR3616647}, the projection onto $\mathcal{P}_{1}$ is given by
\begin{align*}
\Proj_{\mathcal{P}_{1}}(x)=x+\left(\frac{1-\sum_{i=1}^{N}x_{i}}{N}\right)\mathbf{1}\quad\text{for all }x\in\R^{N},
\end{align*}
where $\mathbf{1}=(1,\ldots,1)\in\R^{N}$.

On the other hand, note that $\mathcal{P}_{2}=\bigtimes_{i=1}^{N}\{p\in\R\,:\,p\leq q_{i}\}$. Then by \cite[Proposition~29.3]{MR3616647},
\begin{align*}
(\Proj_{\mathcal{P}_{2}}(x))_{i}=\begin{cases}
x_{i} & \text{if } x_{i}\leq q_{i}\\
q_{i} & \text{if } x_{i}>q_{i}
\end{cases}\quad\text{for all }i\in\{1,\ldots,N\}\text{ and }x\in\R^{N}.
\end{align*}
Finally, the projection onto $\mathcal{P}_{3}$ can be calculated using \cite[Example~29.21]{MR3616647}.
\end{remark}
\subsection{Forward-backward with subspaces}
 Let us see another method to solve the inclusions \eqref{inc_2}, \eqref{inc_2_2}, and \eqref{inc_2_3}. Consider the spaces $\mathsf{H}_{1}=(\H\times \R^{N})^{N+1}$ and $\mathsf{H}_{2}=(\H\times \R^{N})^{N+2}$. Let $\widetilde{A}_{1}\colon\mathsf{H}_{1}\rightarrow 2^{\mathsf{H}_{1}}$, $\widetilde{A}_{2}\colon\mathsf{H}_{2}\rightarrow 2^{\mathsf{H}_{2}}$, $\widetilde{A}_{3}\colon\mathsf{H}_{2}\rightarrow 2^{\mathsf{H}_{2}}$, $\widetilde{C}_{1}\colon\mathsf{H}_{1}\rightarrow\mathsf{H}_{1}$, and $\widetilde{C}_{2}\colon\mathsf{H}_{2}\rightarrow\mathsf{H}_{2}$ the operators defined by
 \begin{align*}
  \widetilde{A}_{1}&=B_{1}\times\cdots\times B_{N}\times A_{1}\\
  \widetilde{C}_{1}&=(C_{1},\underbrace{0,\ldots,0}_{N-\text{times}})\\
 \widetilde{A}_{2}&=B_{1}\times\cdots\times B_{N}\times A_{1}\times A_{2}\\
 \widetilde{A}_{3}&=B_{1}\times\cdots\times B_{N}\times A_{1}\times A_{3}\\
  \widetilde{C}_{2}&=(C_{1},\underbrace{0,\ldots,0}_{(N+1)-\text{times}}).
 \end{align*}
 Then the inclusion \eqref{inc_2} is equivalent to
 \begin{align}
\label{inc_3}
\text{Find }z\in \mathsf{H}_{1}\text{ such that }\quad 0\in \widetilde{A}_{1}(z)+\widetilde{C}_{1}(z)+N_{V_{1}}(z),
 \end{align}
 where $V_{1}=\{z\in \mathsf{H}_{1}\,:\,z_{1}=\ldots=z_{N+1}\}$. Indeed, we have that $V_{1}^{\perp}=\{z\in \mathsf{H}_{1}\,:\,\sum_{i=1}^{N+1}z_{i}=0\}=N_{V_{1}}(z)$ for all $z\in V_{1}$. Now, define $D_{i}=B_{i}$ for all $i\in\{1,\ldots,N\}$, $D_{N+1}=A_{1}$, and $C_{i}=0$ for all $i\in\{2,\ldots,N+1\}$. Thus, denoting $\mathbf{x}=(x,p)$, \eqref{inc_2} is equivalent to
 \begin{align*}
0\in \sum_{i=1}^{N+1}D_{i}(\mathbf{x})+C_{i}(\mathbf{x})&\Leftrightarrow \left(\exists (y_{i})_{i=1}^{N+1}\in\bigtimes_{i=1}^{N+1}D_{i}(\mathbf{x})\right)\quad 0=\sum_{i=1}^{N+1}(-y_{i}-C_{i}(\mathbf{x}))\\
&\Leftrightarrow \left(\exists (y_{i})_{i=1}^{N+1}\in\bigtimes_{i=1}^{N+1}D_{i}(\mathbf{x})\right)\quad -(y_{1},\ldots,y_{N+1})-\widetilde{C}_{1}(\mathbf{x},\ldots,\mathbf{x})\in V_{1}^{\perp}\\
&\Leftrightarrow 0\in \widetilde{A}_{1}(\mathbf{x},\ldots,\mathbf{x})+\widetilde{C}_{1}(\mathbf{x},\ldots,\mathbf{x})+N_{V_{1}}(\mathbf{x},\ldots,\mathbf{x}).
 \end{align*}
 Similarly, the inclusion \eqref{inc_2_2} is equivalent to
 \begin{align}
\label{inc_3_2}
\text{Find }z\in \mathsf{H}_{2}\text{ such that }\quad 0\in \widetilde{A}_{2}(z)+\widetilde{C}_{2}(z)+N_{V_{2}}(z),
 \end{align}
 where $V_{2}=\{z\in \mathsf{H}_{2}\,:\,z_{1}=\ldots=z_{N+2}\}$. Finally, the inclusion \eqref{inc_2_3} is equivalent to
 \begin{align}
\label{inc_3_3}
\text{Find }z\in \mathsf{H}_{2}\text{ such that }\quad 0\in \widetilde{A}_{3}(z)+\widetilde{C}_{2}(z)+N_{V_{2}}(z).
 \end{align}
 \begin{remark}
Note that the projection on $V_{1}$ is given by
\begin{align}
\label{proj_diag}
\Proj_{V_{1}}(z)=\left(\dfrac{1}{N+1}\sum_{j=1}^{N+1}z_{j}\right)_{i=1}^{N+1}\quad\text{for all }z\in\mathsf{H}_{1},
\end{align}
and the projection onto $V_{2}$ is similar.
 \end{remark}
 In order to solve \eqref{inc_3}, \eqref{inc_3_2}, and \eqref{inc_3_3}, we use the algorithm proposed in \cite[Corollary~5.3]{briceno2015forward} when $\lambda_{n}\equiv 1$ (which coincides with the algorithm proposed in \cite{briceno2023primal} in this case). From \cite[Corollary~5.3]{briceno2015forward}, \cite[Proposition~23.18]{MR3616647}, and \eqref{proj_diag}, we deduce the following results.
 \begin{proposition}
\label{prop_fb_sev_1}
Let $\gamma\in\left]0,2\beta\right[$, $z^{0}=((x_{1}^{0},p_{1}^{0}),\ldots,(x_{N+1}^{0},p_{N+1}^{0}))\in V_{1}$, and 

$\overline{z}^{0}=((\overline{x}_{1}^{0},\overline{p}_{1}^{0}),\ldots,(\overline{x}_{N+1}^{0},\overline{p}_{N+1}^{0}))\in V_{1}^{\perp}$. For every $k\in\mathbb{N}$, we consider the following routine.
\begin{equation}
\label{alg_fb_sev_1}
\left\lfloor
\begin{array}{ll}
(\widetilde{x}_{i}^{k},\widetilde{p}_{i}^{k}):=J_{\gamma B_{i}}\left(x_{i}^{k}+\gamma\overline{x}_{i}^{k}-\frac{\gamma}{N+1}\nabla h(x_{1}^{k}),p_{i}^{k}+\gamma\overline{p}_{i}^{k}\right)\textnormal{ for all }i\in\{1,\ldots,N\}\\
(\widetilde{x}_{N+1}^{k},\widetilde{p}_{N+1}^{k}):=\left(\Proj_{\Q}\left(x_{N+1}^{k}+\gamma\overline{x}_{N+1}^{k}-\frac{\gamma}{N+1}\nabla h(x_{1}^{k})\right),\Proj_{\mathcal{P}_{1}}(p_{N+1}^{k}+\gamma \overline{p}_{N+1}^{k})\right)\\
z_{i}^{k+1}:=(x_{i}^{k+1},p_{i}^{k+1})=\frac{1}{N+1}\sum_{j=1}^{N+1}(\widetilde{x}_{j}^{k},\widetilde{p}_{j}^{k})\textnormal{ for all }i\in\{1,\ldots,N+1\}\\
\overline{z}_{i}^{k+1}:=(\overline{x}_{i}^{k+1},\overline{p}_{i}^{k+1})=(\overline{x}_{i}^{k},\overline{p}_{i}^{k})+\frac{1}{\gamma}(x_{i}^{k+1}-\widetilde{x}_{i}^{k},p_{i}^{k+1}-\widetilde{p}_{i}^{k})\textnormal{ for all }i\in\{1,\ldots,N+1\}.
\end{array}
\right.
\end{equation}
Then there exists $z\in V_{1}$ solution to problem~\eqref{inc_3} such that $z^{k}\rightharpoonup z$.
\end{proposition}

\begin{proposition}
\label{prop_fb_sev_2}
Let $\gamma\in\left]0,2\beta\right[$, $z^{0}=((x_{1}^{0},p_{1}^{0}),\ldots,(x_{N+2}^{0},p_{N+2}^{0}))\in V_{2}$, and 

$\overline{z}^{0}=((\overline{x}_{1}^{0},\overline{p}_{1}^{0}),\ldots,(\overline{x}_{N+2}^{0},\overline{p}_{N+2}^{0}))\in V_{2}^{\perp}$. For every $k\in\mathbb{N}$, we consider the following routine.
\begin{equation}
\label{alg_fb_sev_2}
\left\lfloor
\begin{array}{ll}
(\widetilde{x}_{i}^{k},\widetilde{p}_{i}^{k}):=J_{\gamma B_{i}}\left(x_{i}^{k}+\gamma\overline{x}_{i}^{k}-\frac{\gamma}{N+2}\nabla h(x_{1}^{k}),p_{i}^{k}+\gamma\overline{p}_{i}^{k}\right)\textnormal{ for all }i\in\{1,\ldots,N\}\\
(\widetilde{x}_{N+1}^{k},\widetilde{p}_{N+1}^{k}):=\left(\Proj_{\Q}\left(x_{N+1}^{k}+\gamma\overline{x}_{N+1}^{k}-\frac{\gamma}{N+2}\nabla h(x_{1}^{k})\right),\Proj_{\mathcal{P}_{1}}(p_{N+1}^{k}+\gamma \overline{p}_{N+1}^{k})\right)\\
(\widetilde{x}_{N+2}^{k},\widetilde{p}_{N+2}^{k}):=\left(x_{N+2}^{k}+\gamma\overline{x}_{N+2}^{k}-\frac{\gamma}{N+2}\nabla h(x_{1}^{k}),\Proj_{\mathcal{P}}(p_{N+2}^{k}+\gamma \overline{p}_{N+2}^{k})\right)\\
z_{i}^{k+1}:=(x_{i}^{k+1},p_{i}^{k+1})=\frac{1}{N+2}\sum_{j=1}^{N+2}(\widetilde{x}_{j}^{k},\widetilde{p}_{j}^{k})\textnormal{ for all }i\in\{1,\ldots,N+2\}\\
\overline{z}_{i}^{k+1}:=(\overline{x}_{i}^{k+1},\overline{p}_{i}^{k+1})=(\overline{x}_{i}^{k},\overline{p}_{i}^{k})+\frac{1}{\gamma}(x_{i}^{k+1}-\widetilde{x}_{i}^{k},p_{i}^{k+1}-\widetilde{p}_{i}^{k})\textnormal{ for all }i\in\{1,\ldots,N+2\}.
\end{array}
\right.
\end{equation}
Then the following hold.
\begin{enumerate}
    \item If $\mathcal{P}=\mathcal{P}_{2}$, then there exists $z\in V_{2}$ solution to problem~\eqref{inc_3_2} such that $z^{k}\rightharpoonup z$.
    \item If $\mathcal{P}=\mathcal{P}_{3}$, then there exists $z\in V_{2}$ solution to problem~\eqref{inc_3_3} such that $z^{k}\rightharpoonup z$.
\end{enumerate}
\end{proposition}

\subsection{Davis-Yin's Formulation}
Another method for solving the problems~\eqref{inc_3}, \eqref{inc_3_2}, and \eqref{inc_3_3} (and therefore the inclusions \eqref{inc_2}, \eqref{inc_2_2}, and \eqref{inc_2_3}) is the algorithm proposed in \cite{davis2017three} since the normal cone of a vector subspace is a maximally monotone operator. From \cite[Theorem~2.1]{davis2017three}, \cite[Proposition~23.18]{MR3616647}, and \eqref{proj_diag}, we obtain the following results.
\begin{proposition}
\label{prop_dy_1}
Let $\gamma\in\left]0,2\beta\right[$ and $z^{0}=((x_{1}^{0},p_{1}^{0}),\ldots,(x_{N+1}^{0},p_{N+1}^{0}))\in \mathsf{H}_{1}$. For every $k\in\mathbb{N}$, we consider the following routine.
\begin{equation}
\label{alg_dy_1}
\left\lfloor
\begin{array}{ll}
\overline{z}_{i}^{k}:=(\overline{x}_{i}^{k},\overline{p}_{i}^{k})=\frac{1}{N+1}\sum_{j=1}^{N+1}(x_{j}^{k},p_{j}^{k})\textnormal{ for all }i\in\{1,\ldots,N+1\}\\
(\widetilde{x}_{1}^{k},\widetilde{p}_{1}^{k})=J_{\gamma B_{1}}(2\overline{x}_{1}^{k}-x_{1}^{k}-\gamma \nabla h(\overline{x}_{1}^{k}),2\overline{p}_{1}^{k}-p_{1}^{k})\\
(\widetilde{x}_{i}^{k},\widetilde{p}_{i}^{k})=J_{\gamma B_{i}}(2\overline{x}_{i}^{k}-x_{i}^{k},2\overline{p}_{i}^{k}-p_{i}^{k})\quad\textnormal{ for all }i\in\{2,\ldots,N\}\\
(\widetilde{x}_{N+1}^{k},\widetilde{p}_{N+1}^{k})=(\Proj_{\Q}(2\overline{x}_{N+1}^{k}-x_{N+1}^{k}),\Proj_{\mathcal{P}_{1}}(2\overline{p}_{N+1}^{k}-p_{N+1}^{k}))\\
(x_{i}^{k+1},p_{i}^{k+1})=(x_{i}^{k}+\widetilde{x}_{i}^{k}-\overline{x}_{i}^{k},p_{i}^{k}+\widetilde{p}_{i}^{k}-\overline{p}_{i}^{k})\quad\textnormal{ for all }i\in\{1,\ldots,N+1\}.
\end{array}
\right.
\end{equation}
Then there exists $z\in \mathsf{H}_{1}$ solution to problem~\eqref{inc_3} such that $\overline{z}^{k}\rightharpoonup z$.
\end{proposition}

\begin{proposition}
\label{prop_dy_2}
Let $\gamma\in\left]0,2\beta\right[$ and $z^{0}=((x_{1}^{0},p_{1}^{0}),\ldots,(x_{N+2}^{0},p_{N+2}^{0}))\in \mathsf{H}_{2}$. For every $k\in\mathbb{N}$, we consider the following routine.
\begin{equation}
\label{alg_dy_2}
\left\lfloor
\begin{array}{ll}
\overline{z}_{i}^{k}:=(\overline{x}_{i}^{k},\overline{p}_{i}^{k})=\frac{1}{N+2}\sum_{j=1}^{N+2}(x_{j}^{k},p_{j}^{k})\textnormal{ for all }i\in\{1,\ldots,N+2\}\\
(\widetilde{x}_{1}^{k},\widetilde{p}_{1}^{k})=J_{\gamma B_{1}}(2\overline{x}_{1}^{k}-x_{1}^{k}-\gamma \nabla h(\overline{x}_{1}^{k}),2\overline{p}_{1}^{k}-p_{1}^{k})\\
(\widetilde{x}_{i}^{k},\widetilde{p}_{i}^{k})=J_{\gamma B_{i}}(2\overline{x}_{i}^{k}-x_{i}^{k},2\overline{p}_{i}^{k}-p_{i}^{k})\quad\textnormal{ for all }i\in\{2,\ldots,N\}\\
(\widetilde{x}_{N+1}^{k},\widetilde{p}_{N+1}^{k})=(\Proj_{\Q}(2\overline{x}_{N+1}^{k}-x_{N+1}^{k}),\Proj_{\mathcal{P}_{1}}(2\overline{p}_{N+1}^{k}-p_{N+1}^{k}))\\
(\widetilde{x}_{N+2}^{k},\widetilde{p}_{N+2}^{k})=(2\overline{x}_{N+2}^{k}-x_{N+2}^{k},\Proj_{\mathcal{P}}(2\overline{p}_{N+2}^{k}-p_{N+2}^{k}))\\
(x_{i}^{k+1},p_{i}^{k+1})=(x_{i}^{k}+\widetilde{x}_{i}^{k}-\overline{x}_{i}^{k},p_{i}^{k}+\widetilde{p}_{i}^{k}-\overline{p}_{i}^{k})\quad\textnormal{ for all }i\in\{1,\ldots,N+2\}.
\end{array}
\right.
\end{equation}
Then the following hold.
\begin{enumerate}
\item If $\mathcal{P}=\mathcal{P}_{2}$, then there exists $z\in \mathsf{H}_{2}$ solution to problem~\eqref{inc_3_2} such that $\overline{z}^{k}\rightharpoonup z$.
\item If $\mathcal{P}=\mathcal{P}_{3}$, then there exists $z\in \mathsf{H}_{2}$ solution to problem~\eqref{inc_3_3} such that $\overline{z}^{k}\rightharpoonup z$.
\end{enumerate}
\end{proposition}

\subsection{Dual method}
\label{dual_method}
In this section, we use the techniques proposed in \cite[Section~4.2]{sun2021robust} for solving problem~\eqref{prob_principal} for the ambiguity sets considered in the cases~\ref{item1_1}, \ref{item1_2}, and \ref{item1_3}. In each case, we reformulate the dual of the following inner maximization problem
\begin{align}
\label{primal_1}
\max_{p\in\mathcal{P}}h(x)+\sum_{i=1}^{N}p_{i}f_{i}(x)=\sum_{i=1}^{N}p_{i}(f_{i}(x)+h(x)).
\end{align}
Observe that by assumption of each case, we have that the problem \eqref{primal_1} is a convex problem that satisfy the Slater conditions. Therefore, the strong duality hold.

Case~\ref{item1_1}: In this case $\mathcal{P}=\Delta_{N}$ and the Lagrangian function is
\begin{align*}
L(p,\lambda)&=\sum_{i=1}^{N}p_{i}(f_{i}(x)+h(x))+\lambda(1-\sum_{i=1}^{N}p_{i})\\
&=\lambda+\sum_{i=1}^{N}p_{i}(f_{i}(x)+h(x)-\lambda).
\end{align*}
Then, \eqref{primal_1} is equivalent to the following dual problem
\begin{align}
\min_{\lambda\in\R}\max_{p\geq 0}\left\{\lambda+\sum_{i=1}^{N}p_{i}(f_{i}(x)+h(x)-\lambda)\right\}&\Leftrightarrow \min_{\lambda\in\R}\left\{\lambda+\max_{p\geq 0}\sum_{i=1}^{N}p_{i}(f_{i}(x)+h(x)-\lambda)\right\}\nonumber\\
&\Leftrightarrow \min_{\lambda\in\R} \lambda\nonumber\\
&\quad\text{s.t. } f_{i}(x)+h(x)-\lambda\leq 0,\quad\forall i\in\{1,\ldots,N\}.\nonumber
\end{align}
Thus, problem~\eqref{prob_principal} is equivalent to
\begin{align}
&\min_{x\in \H,\,\lambda\in\R}\lambda\label{dual_prob_1}\\
&\text{s.t. } f_{i}(x)+h(x)-\lambda\leq 0,\quad\forall i\in\{1,\ldots,N\}\nonumber\\
&\quad\,\, x\in Q\nonumber
\end{align}
Case~\ref{item1_2}: In this case $\mathcal{P}=\{p\in\Delta_{N}\,:\,p\leq q\}$ and hence the Lagrangian function is
\begin{align*}
L(p,\lambda,\mu)&=\sum_{i=1}^{N}p_{i}(f_{i}(x)+h(x))+\lambda(1-\sum_{i=1}^{N}p_{i})+\sum_{i=1}^{N}\mu_{i}(q_{i}-p_{i})\\
&=\lambda+\mu^{\top}q+\sum_{i=1}^{N}p_{i}(f_{i}(x)+h(x)-\mu_{i}-\lambda).
\end{align*}
Then, from strong duality, \eqref{primal_1} is equivalent to the following dual problem
\begin{align}
&\min_{\substack{\lambda\in\R\\\mu\in\R_{+}^{N}}}\max_{p\geq 0}\left\{\lambda+\mu^{\top}q+\sum_{i=1}^{N}p_{i}(f_{i}(x)+h(x)-\mu_{i}-\lambda)\right\}\nonumber\\
&\Leftrightarrow \min_{\substack{\lambda\in\R\\\mu\in\R_{+}^{N}}} \lambda+\mu^{\top}q\nonumber\\
&\quad\text{s.t. } f_{i}(x)+h(x)-\mu_{i}-\lambda\leq 0,\quad\forall i\in\{1,\ldots,N\}.\nonumber
\end{align}
Thus, problem~\eqref{prob_principal} is equivalent to
\begin{align}
&\min_{\substack{x\in \H,\,\lambda\in\R\\\mu\in\R_{+}^{N}}}\lambda+\mu^{\top}q\label{dual_prob_2}\\
&\text{s.t. } f_{i}(x)+h(x)-\mu_{i}-\lambda\leq 0,\quad\forall i\in\{1,\ldots,N\}\nonumber\\
&\quad\,\, x\in Q\nonumber
\end{align}
Case~\ref{item1_3}: Here $\mathcal{P}=\{p\in\Delta_{N}\,:\,\mu_{-}\leq \sum_{i=1}^{N}p_{i}\xi_{i}\leq\mu_{+}\}$ and hence the Lagrangian function is
\begin{align*}
L(p,\lambda,\beta,\gamma)&=\sum_{i=1}^{N}p_{i}(f_{i}(x)+h(x))+\lambda(1-\sum_{i=1}^{N}p_{i})+\beta(\mu_{+}-\sum_{i=1}^{N}p_{i}\xi_{i})+\gamma(-\mu_{-}+\sum_{i=1}^{N}p_{i}\xi_{i})\\
&=\lambda+\beta\mu_{+}-\gamma\mu_{-}+\sum_{i=1}^{N}p_{i}(f_{i}(x)+h(x)-\lambda+(\gamma-\beta)\xi_{i}).
\end{align*}
Therefore, from strong duality, \eqref{primal_1} is equivalent to the following dual problem
\begin{align}
&\min_{\substack{\lambda\in\R\\\beta\geq 0,\gamma\geq 0}}\max_{p\geq 0}\left\{\lambda+\beta\mu_{+}-\gamma\mu_{-}+\sum_{i=1}^{N}p_{i}(f_{i}(x)+h(x)-\lambda+(\gamma-\beta)\xi_{i})\right\}\nonumber\\
&\Leftrightarrow \min_{\substack{\lambda\in\R\\\beta\geq 0,\gamma\geq 0}} \lambda+\beta\mu_{+}-\gamma\mu_{-}\nonumber\\
&\quad\text{s.t. } f_{i}(x)+h(x)-\lambda+\gamma\xi_{i}-\beta\xi_{i}\leq 0,\quad\forall i\in\{1,\ldots,N\}.\nonumber
\end{align}
Thus, problem~\eqref{prob_principal} is equivalent to
\begin{align}
&\min_{\substack{x\in \H,\,\lambda\in\R\\\beta\geq 0,\gamma\geq 0}}\lambda+\beta\mu_{+}-\gamma\mu_{-}\label{dual_prob_3}\\
&\text{s.t. } f_{i}(x)+h(x)-\lambda+\gamma\xi_{i}-\beta\xi_{i}\leq 0,\quad\forall i\in\{1,\ldots,N\}\nonumber\\
 &\quad\,\, x\in Q\nonumber
\end{align}

 \section{Applications}
 \label{sec_app}
\subsection{Couette inverse problem}
 Consider a rheometer of coaxial cylinders with fluid. Let $\Omega\colon\R\rightarrow\R$ representing the angular velocity of the inner cylinder. The function $\Omega$ has the following form \cite{krieger1952direct}:
 \begin{align*}
    \Omega(\tau)=\dfrac{1}{2}\int_{\beta\tau}^{\tau}\dfrac{f(t)}{t}dt,
 \end{align*}
 where $\tau$ is the tangential shear stress on the cylindrical surface, $f\colon\R\rightarrow\R$ is the rheological curve of the fluid, and $\beta<1$. Given $\Omega\colon\R\rightarrow\R$, the Couette inverse problem consists in to find the function $f$. Suppose that we have a sample of size $r$ of values $(\Omega_{i},\tau_{i})_{i=1}^{r}$ representing measurements obtained from the rheometer. Assume that the function $f$ has the form $f(t)=\sum_{j=1}^{\ell}a_{j}e_{j}(t)$, where $a_{j}\in\R$ and $e_{j}$ is a function such that $t\mapsto\frac{e_{j}(t)}{t}$ is integrable for all $j\in\{1,\ldots,\ell\}$. Given $(\Omega_{i},\tau_{i})_{i=1}^{r}$, the problem is to find the constants $(a_{j})_{j=1}^{\ell}$ such that 
 \begin{align*}
\Omega_{i}=\sum_{j=1}^{\ell}a_{j}\cdot\dfrac{1}{2}\int_{\beta\tau_{i}}^{\tau_{i}}\dfrac{e_{j}(t)}{t}dt,\quad\text{for all }i\in\{1,\ldots,r\},
 \end{align*}
 which is equivalent to $\Omega=Ax$, where $\Omega=(\Omega_{i})_{i=1}^{r}\in\R^{r}$, $A\in\R^{r\times\ell}$ is the matrix defined by $A_{ij}=\dfrac{1}{2}\displaystyle\int_{\beta\tau_{i}}^{\tau_{i}}\frac{e_{j}(t)}{t}dt$, and $x=(a_{j})_{j=1}^{\ell}$. It is possible that the system $Ax=\Omega$ has no solutions or that it has several solutions. In these cases, a common approach to finding an approximate solution is to solve the following regularized least squares problem
 \begin{align}
 \label{lq_reg_1}
\min_{x\in\R^{\ell}} \|Ax-\Omega\|^{2}+\lambda R(x),
 \end{align} 
 where $R$ is the regularization function which represents a prior information on $x$ and $\lambda>0$ is the regularization parameter. In the case when $R(x)=\|Dx\|^{2}$, where $D\in\R^{p\times\ell}$ and assuming that $A^{\top}A+\lambda D^{\top}D$ is positive definite (for example, if $D=\Id$), the Fermat's rule implies that the solution to \eqref{lq_reg_1} is
 \begin{align*}
    \overline{x}=(A^{\top}A+\lambda D^{\top}D)^{-1}A^{\top}\Omega.
 \end{align*}
  Suppose now that we have $N$ samples of size $r$ of values $(\Omega^{1}_{i})_{i=1}^{r},\ldots,(\Omega^{N}_{i})_{i=1}^{r}$. Let $\xi\colon\overline{\Omega}\rightarrow\Xi$ the random vector defined on a measurable space $(\overline{\Omega},\mathcal{A})$ whose values are in the finite set $\Xi=\{\Omega^{1},\ldots,\Omega^{N}\}$ and let $\mathcal{P}\subset\Delta_{N}$ a nonempty closed convex subset of probability measures on the space $(\overline{\Omega},\mathcal{A})$ supported on $\Xi$. For every $\mathbb{P}\in\mathcal{P}$ and $k\in\{1,\ldots,N\}$, we denote $p_{k}=\mathbb{P}(\{\omega\in\overline{\Omega}\,:\,\xi(\omega)=\Omega^{k}\})$. Thus, we consider the following DRO problem
  \begin{align}
\label{dro_lq_reg_1}
\min_{x\in\R^{\ell}}\sup_{p\in\mathcal{P}}\sum_{k=1}^{N}p_{k}\|Ax-\Omega^{k}\|^{2}+\lambda\|Dx\|^{2}.
  \end{align}
  Note that, since $\mathcal{P}\subset\Delta_{N}$, problem \eqref{dro_lq_reg_1} is equivalent to
  \begin{align*}
\min_{x\in\R^{\ell}}\|Ax\|^{2}+\lambda\|Dx\|^{2}+\sup_{p\in\mathcal{P}}\sum_{k=1}^{N}p_{k}(-2x^{\top}A^{\top}\Omega^{k}+\|\Omega^{k}\|^{2}),
  \end{align*}
  which is equivalent to problem~\eqref{prob_item4} when $\Q=\R^{\ell}$, $h(x)=\|Ax\|^{2}+\lambda\|Dx\|^{2}$ (which has  Lipschitz gradient), $a_{k}=-2A^{\top}\Omega^{k}$, and $\xi_{k}=\|\Omega^{k}\|^{2}$.
  
  Note that if $D=\Id$, then the problem~\eqref{dro_lq_reg_1} reduces to
  \begin{align}
\label{dro_lq_reg_2}
\min_{x\in\R^{\ell}}\dfrac{1}{2\mu}\|x\|^{2}+ \sup_{p\in\mathcal{P}}\sum_{k=1}^{N}p_{k}(\langle x,Qx\rangle+\langle b_{k},x\rangle+c_{k}),
  \end{align}
  where $Q=A^{\top}A$, $b_{k}=-2A^{\top}\Omega^{k}$, $c_{k}=\|\Omega^{k}\|^{2}$, and $\mu=\frac{1}{2\lambda}$. The problem~\eqref{dro_lq_reg_2} is equivalent to find $P_{\mu}\widetilde{f}(0)$, where $\widetilde{f}$ is defined in \eqref{sup_fun}. Thus, \eqref{dro_lq_reg_2} can be solved using the algorithm proposed in Proposition~\ref{calc_prox_sup_1}.
  \subsection{Denoising}
  Suppose that we have a noisy measurement of a signal $x\in\R^{n}$:
\begin{align*}
b=x+w,
\end{align*}
where $x$ is an unknown signal, $w$ is an unknown noise vector, and $b$ is the known measurement
vector. The denoising problem is the following: Given $b$, find a good estimate of $x$. The regularized least squares problem associated with this problem is
\begin{align}
\label{den_reg_1}
\min_{x\in\R^{n}} \|x-b\|^{2}+\lambda R(x),
\end{align}
where $R(x)$ is a regularization term which represents some a priori information on the signal and $\lambda>0$ is a given regularization parameter. Typically, it is considered a regularization function of the form
\begin{align*}
R(x)=\sum_{i=1}^{n-1}(x_{i}-x_{i+1})^{2}.
\end{align*}
Note that $R(x)=\|Lx\|^{2}$, where $L\in\R^{(n-1)\times n}$ is given by
\begin{align*}
L=\begin{pmatrix}
    1 & -1 & 0 & 0 & \cdots & 0 & 0\\
    0 &  1 & -1 & 0 &\cdots & 0 & 0\\
    0 & 0 & 1 & -1 &\cdots & 0 & 0\\
    \vdots & \vdots & \vdots &\vdots & & \vdots &\vdots\\
    0 & 0 & 0 & 0 & \cdots & 1 & -1
\end{pmatrix}.
\end{align*}
In this case, by the Fermat's rule, we have that the solution to \eqref{den_reg_1} is given by
\begin{align*}
\overline{x}=(\Id+\lambda L^{\top}L)^{-1}b.
\end{align*}
Suppose that we have $N$ noisy measurement $b^{1},\ldots,b^{N}$:
\begin{align*}
b^{i}=x^{i}+w^{i},
\end{align*}
where, for all $i\in\{1,\ldots,N\}$, $x^{i}\in\R^{n}$ is an unknown signal and $w^{i}$ is an unknown noise vector. The objective is to find a good estimate of $x^{i}$ for all $i\in\{1,\ldots,N\}$ in order to minimize the largest norm of the noise vectors. Thus, we consider the following robust optimization problem
\begin{align}
\label{dro_quad_sep}
\min_{\mathbf{x}=(x^{1},\ldots,x^{N})\in\R^{nN}}\left\{\lambda\sum_{j=1}^{N}\|L_{j}x^{j}\|^{2}+\max_{1\leq i\leq N}\|x^{i}-b^{i}\|^{2}\right\},
\end{align}
where $\|L_{j}x^{j}\|^{2}$ represents the regularization term of the variable $x^{j}$. The problem~\eqref{dro_quad_sep} is particular case of problem~\eqref{prob_item3} when $V=\Q_{i}=\H=\R^{n}$, $H(\mathbf{x})=\lambda\sum_{j=1}^{N}\|L_{j}x^{j}\|^{2}$, and $\xi_{i}=b^{i}$. Note that $\nabla H(\mathbf{x})=2\lambda(L_{1}^{\top}L_{1}x^{1},\ldots,L_{N}^{\top}L_{N}x^{N})$. Then, for every $\mathbf{x}$ and $\mathbf{y}$ in $\R^{nN}$, we have
\begin{align*}
\|\nabla H(\mathbf{x})-\nabla H(\mathbf{y})\|^{2}&=\sum_{j=1}
^{N}\|2\lambda L_{j}^{\top}L_{j}(x^{j}-y^{j})\|^{2}\\
&\leq (2\lambda)^{2}\max_{1\leq j\leq N}\|L_{j}^{\top}L_{j}\|^{2}\|\mathbf{x}-\mathbf{y}\|^{2}.
\end{align*}
Hence $\nabla H$ is $\beta^{-1}$-Lipschitz with $\beta^{-1}=2\lambda\displaystyle\max_{1\leq j\leq N}\|L_{j}^{\top}L_{j}\|$.
\section{Numerical experiments}
\label{sec_num_exp}
Consider the problem~\eqref{prob_principal} when $\H=\R^{n}$, $f_{i}(x)=\langle a_{i},x\rangle+\xi_{i}$ for all $i\in\{1,\ldots,N\}$, and $\Q=\{x\in\R^{n}\,:\,Ax=b\}$, where $a_{i}\in\R^{n}\backslash\{0\}$, $\xi_{i}\in\R$, $A\in\R^{m\times n}$ satisfies $\ker A^{\top}=\{0\}$, and $b\in\ran(A)$. We consider two cases for the function $h$. First, we consider $h(x)=\frac{1}{2}x^{\top}Mx$, where $M\in\R^{n\times n}$ is symmetric positive definite. In this case, the problem~\eqref{prob_principal} is
\begin{align}
\label{prob_num_exp_1}
\min_{x\in \Q}\left\{\frac{1}{2}x^{\top}Mx+\sup_{p\in\mathcal{P}}\sum_{i=1}^{N}p_{i}(\langle a_{i},x\rangle+\xi_{i})\right\}.
\end{align}
Second, we consider $h(x)=c^{\top}x$, where $c\in\R^{n}$. In this case, problem~\eqref{prob_principal} is
\begin{align}
\label{prob_num_exp_3}
\min_{x\in\Q}\left\{c^{\top}x+\sup_{p\in\mathcal{P}}\sum_{i=1}^{N}p_{i}(\langle a_{i},x\rangle+\xi_{i})\right\}.
\end{align}
We consider two ambiguity sets $\mathcal{P}$. First, we consider $\mathcal{P}=\Delta_{N}$.
Second, we consider $\mathcal{P}=\widetilde{\mathcal{P}}:=\{p\in\Delta_{N}\,:\,\mu_{-}\leq\sum_{i=1}^{N}p_{i}\xi_{i}\leq\mu_{+}\}$, where $\xi_{i}\in\R$, $\mu_{-}\in\R$, and $\mu_{+}\in\R$ satisfies the conditions in \ref{item1_3}.

Note that, since $M$ is positive definite, then $x\mapsto \frac{1}{2}x^{\top}Mx$ is coercive and strictly convex. Thus, by \cite[Proposition~11.15(i)]{MR3616647}, the problem \eqref{prob_num_exp_1} have an unique solution. On the other hand, we assume that problem~\eqref{prob_num_exp_3} has solutions.

For solving problems~\eqref{prob_num_exp_1} and \eqref{prob_num_exp_3} with $\mathcal{P}=\Delta_{N}$, we use the algorithms \eqref{alg_pi_2} (with $V=\mathcal{D}$, $\Q_{i}=\Q$, and $H(\mathbf{x})=h(x_{1})$), \eqref{alg_aragon_1}, \eqref{alg_fb_sev_1}, and \eqref{alg_dy_1}, while that for solving problems~\eqref{prob_num_exp_1} and \eqref{prob_num_exp_3} with $\mathcal{P}=\widetilde{\mathcal{P}}$, we use the algorithms \eqref{alg_pi_2}, \eqref{alg_aragon_2}, \eqref{alg_fb_sev_2}, and \eqref{alg_dy_2}. Note that in the context of problem~\eqref{prob_num_exp_3}, the equivalent formulations \eqref{dual_prob_1} and \eqref{dual_prob_3} reduce to a linear program, which can be solved using the function \texttt{linprog} of MATLAB. Thus, also we use the equivalent formulation~\eqref{dual_prob_1} for solving \eqref{prob_num_exp_3} when $\mathcal{P}=\Delta_{N}$ and we use the formulation~\eqref{dual_prob_3} for solving \eqref{prob_num_exp_3} when $\mathcal{P}=\widetilde{\mathcal{P}}$. For each method, we obtain the average execution time (in seconds) and the average number of iterations from $20$ random instances for the matrices $A$ and $M$, the vectors $\{a_{i}\}$, $b$, and $c$, and the scalars $\{\xi_{i}\}$, $\mu_{-}$, and $\mu_{+}$ (we consider $\mu_{-}\in\left[0,1/2\right]$ and $\mu_{+}\in\left[1/2,1\right]$). We
measure the efficiency for different values of $N$, $n$, and $m$. We choose $n=m$. We label the algorithm in \eqref{alg_pi_2} as \textit{prox max}, algorithms in \eqref{alg_aragon_1} and \eqref{alg_aragon_2} as \textit{distributed FB}, algorithms in \eqref{alg_fb_sev_1} and \eqref{alg_fb_sev_2}
as \textit{FB with subspaces}, and algorithms in \eqref{alg_dy_1} and \eqref{alg_dy_2} as \textit{Davis-Yin}. In addition, we label the method proposed in Section~\ref{dual_method} for solving \eqref{prob_num_exp_3} as \textit{dual method}. We choose to stop
every algorithm when the norm of the difference between two consecutive iterations
is less than $10^{-5}$ or the number of iterations exceeds $30000$. The results are the following.


\begin{table}[H]
\caption{Average execution time (number of iterations) with relative error tolerance $e=10^{-5}$} for solving problem~\eqref{prob_num_exp_1} with $\mathcal{P}=\Delta_{N}$.
\centering
\begin{tabular}{|c|c|c|c|c|}
\hline
$(n,m,N)$ & prox max & distributed FB & FB with subspaces & Davis-Yin\\ \hline
 $(100,100,10)$&  2.224 (1137) &  4.723 (3806) & 7.892 (5880) &  7.497 (5530) \\ \hline
 $(100,100,50)$&  3.163 (580) & 17.317 (12849) & 6.260 (3493) & 5.377 (3314) \\ \hline
 $(100,100,100)$ &  11.728 (1160) & 67.524 (23416) & 6.056 (2018) & 6.828 (2086) \\ \hline
 $(200,200,50)$& 9.962 (580) & 19.387 (12835) & 2.231 (498) & 1.539 (937) \\ \hline
 $(200,200,100)$& 41.647 (1160) & 100.361 (23404) & 7.065 (725) & 8.887 (1722) \\ \hline
\end{tabular}
\label{tab_1_1}
\end{table}

\begin{table}[H]
\caption{Average execution time (number of iterations) with relative error tolerance $e=10^{-5}$} for solving problem~\eqref{prob_num_exp_1} with $\mathcal{P}=\widetilde{\mathcal{P}}$.
\centering
\begin{tabular}{|c|c|c|c|c|}
\hline
$(n,m,N)$ & prox max & distributed FB & FB with subspaces & Davis-Yin\\ \hline
 $(100,100,10)$&  2.277 (1160) & 5.417 (4057) & 7.043 (4799) &  6.242 (4304) \\ \hline
 $(100,100,50)$&  4.663 (580) & 18.313 (13065) & 6.312 (3055) & 4.666 (2899) \\ \hline
 $(100,100,100)$ & 24.465 (1160) & 67.241 (23624) & 5.823 (1915) & 6.311 (2012) \\ \hline
 $(200,200,50)$& 10.312 (580) & 20.348 (13050) & 2.505 (542) & 1.591 (948) \\ \hline
 $(200,200,100)$& 46.704 (1160) & 86.954 (23613) & 8.723 (868) & 8.092 (1737) \\ \hline
\end{tabular}
\label{tab_1_3}
\end{table}

\begin{table}[H]
\caption{Average execution time (number of iterations) with relative error tolerance $e=10^{-5}$} for solving problem~\eqref{prob_num_exp_3} with $\mathcal{P}=\Delta_{N}$.
\centering
\begin{tabular}{|c|c|c|c|c|c|}
\hline
$(n,m,N)$ & prox max & distributed FB & FB with subsp. & Davis-Yin & Dual method\\ \hline
 $(100,100,10)$&   0.028 (11) &  1.034 (813) &  0.403 (277) &   2.661 (1945) & 0.057  \\ \hline
 $(100,100,50)$&   0.065 (11) & 13.310 (8586) & 1.923 (1051) & 13.107 (7515) & 0.064 \\ \hline
 $(100,100,100)$ & 0.113 (10) & 7.387 (2317) & 5.189 (1943) & 22.904 (7108) & 0.059 \\ \hline
 $(200,200,50)$& 0.164 (10) & 1.920 (1256) & 1.676 (1112) & 8.792 (5632) & 0.093 \\ \hline
 $(200,200,100)$& 0.345 (11) & 9.393 (2664) & 7.844 (2022) & 27.166 (6459) & 0.105 \\ \hline
\end{tabular}
\label{tab_2_1}
\end{table}

\begin{table}[H]
\caption{Average execution time (number of iterations) with relative error tolerance $e=10^{-5}$} for solving problem~\eqref{prob_num_exp_3} with $\mathcal{P}=\widetilde{\mathcal{P}}$.
\centering
\begin{tabular}{|c|c|c|c|c|c|}
\hline
$(n,m,N)$ & prox max & distributed FB & FB with subsp. & Davis-Yin & Dual method\\ \hline
 $(100,100,10)$& 0.040 (11) & 1.128 (864) & 0.463 (301) &  2.907 (2060) & 0.058 \\ \hline
 $(100,100,50)$&  0.204 (11) & 10.441 (7558) & 1.914 (1072) & 11.944 (6647) & 0.059\\ \hline
 $(100,100,100)$ & 1.166 (10) & 8.750 (2993) & 5.130 (1965) & 17.486 (5607) & 0.061\\ \hline
 $(200,200,50)$& 0.256 (10) & 3.659 (2328) & 1.702 (1140) & 6.885 (4257) & 0.092\\ \hline
 $(200,200,100)$& 0.850 (10) & 13.278 (3726) & 7.992 (2041) & 26.853 (5857) & 0.107 \\ \hline
\end{tabular}
\label{tab_2_3}
\end{table}

With respect to the problem \eqref{prob_num_exp_1}, we observe that in the cases $(n,m,N)=(100,100,10)$ and $(n,m,N)=(100,100,50)$, the most efficient algorithm is the proximal algorithm, where in the case of problem~\eqref{prob_num_exp_1} with $\mathcal{P}=\widetilde{\mathcal{P}}$ and $(n,m,N)=(100,100,50)$ the Davis-Yin's method has a similar performance with the proximal method. Now, for the problem~\eqref{prob_num_exp_1} with $\mathcal{P}=\Delta_{N}$, the forward-backward method with subspaces is the most efficient in the cases $(n,m,N)=(100,100,100)$ and $(n,m,N)=(200,200,100)$, while that in the case $(n,m,N)=(200,200,50)$ the fastest method is the Davis-Yin's formulation. On the other hand, for the problem~\eqref{prob_num_exp_1} with $\mathcal{P}=\widetilde{\mathcal{P}}$, in the case $(n,m,N)=(100,100,100)$ the forward-backward algorithm with subspaces is the most efficient, while that in the case $(n,m)=(200,200)$ the Davis-Yin's method has the best performance.

With respect to the problem~\eqref{prob_num_exp_3}, we note that in the case $(n,m,N)=(100,100,10)$ the proximal algorithm is the most efficient for both ambiguity sets, while that in the other cases the dual method is the most efficient. Note that in the case $(n,m,N)=(100,100,50)$ with $\mathcal{P}=\Delta_{N}$, the proximal algorithm and the dual method have similar performance. In addition, the forward backward algorithm with subspaces also is an efficient alternative for solving problem~\eqref{prob_num_exp_3}.
\section{Conclusions}
\label{sec_conc}
In this paper, we provide different splitting algorithms for solving the discrete version of the distributionally robust optimization problem. This problem includes a supremum function in the objective function. The first method is based on calculating the proximity operator of the supremum function. In some cases, we propose an algorithm that converges to the proximity operator, while in a particular case, we provide a closed form for the proximity operator. On the other hand, under qualifications conditions, we prove that the problem is equivalent to solving a monotone inclusion that involves the sum of finitely many monotone operators, and we compute the resolvent of the monotone operators involved in the inclusion. The second method uses an algorithm specialized for that type of inclusion, which is proposed in \cite{aragon2023distributed}. Additionally, we reformulate the inclusion as one that involves two monotone operators and the normal cone to a vector subspace. For solving this reformulation, we use the algorithm proposed in \cite{briceno2015forward} (third method) and the algorithm proposed in \cite{davis2017three} (fourth method).

The proposed algorithms can be applied to solve the Couette inverse problem with uncertainty and the denoising problem with uncertainty. In addition, we prove the efficiency of the algorithms in two particular problems. The first numerical experiment shows that the algorithm which computes the proximity of the supremum function (proximal algorithm) is the most efficient when the dimension of the problem and the size of the uncertainty set are small, whereas the third and fourth methods are more efficient when the dimension of the problem is larger. On the other hand, the second numerical experiment shows that the proximal algorithm is more efficient when the uncertainty set is small, while the method proposed in the literature performs better in other cases. 

\end{document}